\documentclass{article}
\usepackage[margin=1in]{geometry}

\title{Kernel ridge vs.~principal component regression: \\
minimax bounds and adaptability of regularization operators}
\author{Lee H. Dicker\thanks{Rutgers University, \url{ldicker@stat.rutgers.edu}}
\and Dean P. Foster\thanks{Amazon NYC, \url{dean@foster.net}}
\and Daniel Hsu\thanks{Columbia University, \url{djhsu@cs.columbia.edu}}}

\usepackage[numbers, sort&compress]{natbib}

\usepackage{amssymb,verbatim,hyperref}

\usepackage{enumitem}

\newcommand\hide[1]{}

\usepackage{amsmath,amsthm,amsbsy,amsfonts,color,dsfont,mathrsfs,mleftright,commath}

\def\ddefloop#1{\ifx\ddefloop#1\else\ddef{#1}\expandafter\ddefloop\fi}
\def\ddef#1{\expandafter\def\csname bb#1\endcsname{\ensuremath{\mathbb{#1}}}}
\ddefloop ABCDEFGHIJKLMNOPQRSTUVWXYZ\ddefloop

\def\ddef#1{\expandafter\def\csname c#1\endcsname{\ensuremath{\mathcal{#1}}}}
\ddefloop ABCDEFGHIJKLMNOPQRSTUVWXYZ\ddefloop

\def\ddef#1{\expandafter\def\csname bf#1\endcsname{\ensuremath{\mathbf{#1}}}}
\ddefloop ABCDEFGHIJKLMNOPQRSTUVWXYZabcdefghijklmnopqrstuvwxyz\ddefloop
\def\ddef#1{\expandafter\def\csname bf#1\endcsname{\ensuremath{\boldsymbol{\csname #1\endcsname}}}}
\ddefloop {alpha}{beta}{gamma}{delta}{epsilon}{varepsilon}{zeta}{eta}{theta}{vartheta}{iota}{kappa}{lambda}{mu}{nu}{xi}{pi}{varpi}{rho}{varrho}{sigma}{varsigma}{tau}{upsilon}{phi}{varphi}{chi}{psi}{omega}{Gamma}{Delta}{Theta}{Lambda}{Xi}{Pi}{Sigma}{varSigma}{Upsilon}{Phi}{Psi}{Omega}{ell}\ddefloop

\DeclareMathOperator{\diag}{diag}

\newcommand\parens[1]{(#1)}
\renewcommand\norm[1]{\|#1\|} 
\newcommand\braces[1]{\{#1\}}
\newcommand\brackets[1]{[#1]}

\renewcommand\abs[1]{|#1|} 
\newcommand\ind[1]{\ensuremath{\mathds{1}\{#1\}}}
\newcommand\dotp[1]{\langle #1 \rangle}

\newcommand\Parens[1]{\mleft(#1\mright)}
\newcommand\Norm[1]{\mleft\|#1\mright\|}
\newcommand\Braces[1]{\mleft\{#1\mright\}}
\newcommand\Brackets[1]{\mleft[#1\mright]}

\newcommand\Abs[1]{\mleft|#1\mright|}

\newcommand\Dotp[1]{\mleft\langle#1\mright\rangle}

\newcommand\bias{\ensuremath{\operatorname{B}_{\rho}}}
\newcommand\variance{\ensuremath{\operatorname{V}_{\rho}}}

\newcommand\dl{\ensuremath{d_{\lambda}}}

\newcommand\marginal{{\ensuremath{\rho_{\mathcal{X}}}}}
\newcommand\risk{{\ensuremath{\mathcal{R}_{\rho}}}}
\newcommand\reg{{\ensuremath{f^{\dagger}}}}
\newcommand\est{{\ensuremath{\hat f}}}
\newcommand\estl{{\ensuremath{\hat f_{\lambda}}}}
\newcommand\KRR{{\operatorname{KRR}}}
\newcommand\KPCR{{\operatorname{KPCR}}}
\newcommand\krr{{\ensuremath{\hat f_{\operatorname{KRR},\lambda}}}}
\newcommand\kpcr{{\ensuremath{\hat f_{\operatorname{KPCR},\lambda}}}}
\newcommand\linv{{\ensuremath{g_{\lambda}}}}

\newcommand\sinv{{\ensuremath{s_{\lambda}}}}

\let\oldtop\top
\renewcommand\top{{\ensuremath{\scriptscriptstyle{\oldtop}}}}

\def\<{\langle}
\def\>{\rangle}
\def\A{\mathbf{A}}
\def\a{\alpha}

\def\B{\mathbf{B}}
\def\b{\beta}

\def\bD{\mathbf{D}}
\def\ba{\boldsymbol{\alpha}}
\def\bb{\boldsymbol{\beta}}

\def\bphi{\boldsymbol{\phi}}
\def\bPhi{\boldsymbol{\Phi}}

\def\bPsi{\boldsymbol{\Psi}}

\def\d{\delta}
\def\E{\mathbb{E}}
\def\e{\epsilon}
\def\ee{\boldsymbol{\epsilon}}

\def\g{\gamma}

\def\H{\mathcal{H}}
\def\I{\mathbf{I}}
\def\k{\kappa}
\def\K{\mathbf{K}}
\def\l{\lambda}

\def\N{\mathbb{N}}

\def\R{\mathbb{R}}
\def\s{\sigma}
\def\S{\boldsymbol{\varSigma}}
\def\th{\theta}
\def\tr{\mathrm{tr}}
\def\T{\mathbf{T}}
\def\Th{\Theta}

\def\U{\mathbf{U}}

\def\Var{\mathrm{Var}}

\def\x{\mathbf{x}}
\def\X{\mathbf{X}}
\def\y{\mathbf{y}}
\def\Y{\mathbf{Y}}
\def\z{\mathbf{z}}

\newtheorem{theorem}{Theorem}
\newtheorem{lemma}{Lemma}
\newtheorem{proposition}{Proposition}
\newtheorem{remark}{Remark}
\newtheorem{corollary}{Corollary}

\usepackage[textsize=tiny]{todonotes}

\begin{document}

\maketitle

\begin{abstract}
Regularization is an essential element of virtually all kernel methods for
nonparametric regression problems.  A critical factor in the
effectiveness of a given kernel method
is the type of regularization that is employed.
This article compares and contrasts members from a general class of regularization
techniques, which notably includes ridge regression and principal
component regression.  We derive an explicit finite-sample risk bound
for regularization-based estimators
that simultaneously accounts for (i) the structure of the ambient function space,
(ii) the regularity of the true regression function, and (iii) the
adaptability (or {\em qualification}) of the
regularization.  A simple consequence of this upper bound is that the
risk of the regularization-based estimators matches the minimax rate
in a variety of settings.  The general bound also illustrates how some
regularization techniques are more adaptable
than others to favorable regularity properties that the true
regression function may possess.
This, in particular, demonstrates a striking difference between kernel
ridge regression and kernel principal component regression.  Our
theoretical results are supported by numerical experiments.  
\end{abstract}

\section{Introduction}\label{sec:intro}

Suppose that the observed data consists of $\z_i = (y_i,\x_i)$, $i = 1,\dotsc,n$, where $y_i \in \cY
\subseteq \R$ and $\x_i \in \cX \subseteq \R^d$.
Suppose further that $\z_1,\dotsc,\z_n \sim \rho$ are iid from some probability distribution $\rho$
on $\cY \times \cX$.
Let $\rho(\cdot\mid\x)$ denote the conditional distribution of $y_i$ given $\x_i = \x \in \cX$ and
let $\marginal$ denote the marginal distribution of $\x_i$.
Our goal is to use the available data to estimate the regression function of $y$ on $\x$,
\[
  \reg(\x) \ = \ \int_{\cY} y \ \dif\rho(y\mid \x)
  \,,
\]
which minimizes the mean-squared prediction error
\[
  \int_{\cY\times \cX}
  \Parens{ y - f(\x) }^2 \ \dif\rho(y,\x)
\]
over $\marginal$-measurable functions $f \colon \cX \to \R$.
More specifically, for an estimator $\est$ define the risk 
\begin{align}\label{risk}
  \risk(\est)
  & \ = \ \E\Brackets{ \int \Parens{ \reg(\x) - \est(\x) }^2 \ \dif\marginal(\x) }
  \ = \ \E\Brackets{ \, \norm{\reg - \est}^2_{\marginal} } \,,
\end{align}
where the expectation is computed over $\z_1,\dotsc,\z_n$, and ${\norm{\cdot}_{\marginal}}$ denotes
the norm on $L^2(\marginal)$; we seek estimators $\est$ which minimize $\risk(\est)$.  

This is a version of the random design nonparametric regression problem.
There is a vast literature on nonparametric regression, along with a huge variety of corresponding
methods~\citep[e.g.,][]{gyorfi2002distribution,wasserman2006all}.
In this paper, we focus on regularization and kernel methods for estimating $\reg$.\footnote{Here, we mean
``kernel'' as in \emph{reproducing kernel Hilbert space}, rather than
\emph{kernel-smoothing}, which is another popular approach to nonparametric regression.}
Most of our results apply to general regularization operators.
However, our motivating examples are two well-known regularization techniques: Kernel ridge
regression (which we refer to as ``KRR''; KRR is also known as Tikhonov regularization) and kernel
principal component regression (``KPCR''; also known as spectral cut-off regularization).  

Our main theorem is a new upper bound on the risk of a general class of kernel-regularization
methods, which includes both KRR and KPCR (Theorem~\ref{thm:UB}).
The theorem substantially generalizes previously published bounds (see Section \ref{sec:related} for
a discussion of related work) and illustrates the dependence of the risk on three important
features: (i) the structure of the ambient reproducing kernel Hilbert
space (RKHS), (ii) the specific regularization technique
employed, and (iii) the regularity (often interpreted as smoothness) of the function to be
estimated.
One
consequence of the theorem is that the regularization methods studied in this paper
(including KRR and KPCR) achieve the minimax rate for estimating $\reg$ in a variety of settings.
A second consequence is that certain regularization methods (including
KPCR, but not KRR) may adapt to
favorable regularity of $\reg$ to attain even faster convergence rates, while others (notably KRR)
are limited in this regard due to a well-known \emph{saturation}
effect~\citep{neubauer1997converse,mathe2004saturation,bauer2007regularization}.
This illustrates a striking advantage that KPCR may have over KRR in these settings.  

\section{Related work}\label{sec:related}

Kernel ridge regression has been studied extensively in the literature.
Indeed, bounds for KRR that are similar to our Theorem \ref{thm:UB} have been derived by
\citet{caponnetto2007optimal}.
Moreover, it is well-known that KRR is minimax in many of the settings considered in this paper,
such as those described in Corollaries \ref{cor:poly}--\ref{cor:gauss}~\citep{caponnetto2007optimal, zhang2005learning, zhang2013divide}.
However, these cited results apply only to KRR, while the results in this paper apply to a
substantially larger class of regularization operators (including, for example, KPCR).  

Beyond KRR, there has also been significant research into more general regularization methods, like
those considered in this paper.
However, our bounds are sharper than previously published results on general regularization
operators.
For instance, unlike our Theorem \ref{thm:UB}, the bounds of \citet{bauer2007regularization} do not
illustrate the dependence of the risk on the ambient Hilbert space.
Thus, while our approach immediately implies that many of the regularization methods under
consideration are minimax optimal, it seems difficult (if not impossible) to draw this conclusion
using the approach of \citeauthor{bauer2007regularization}.
General regularization operators are studied by~\citet{caponnetto2010cross}, but their results
require a semi-supervised setting where an additional pool of unlabeled data is available.  

One of the major practical implications of this paper is that KPCR may have significant advantages
over KRR in some settings.
This has been observed previously by other researchers; others have even noted that Tikhonov
regularization (KRR) saturates, while spectral cut-off regularization (KPCR) does not
\citep{mathe2004saturation, bauer2007regularization, logerfo2008spectral}.
Our results (Theorem \ref{thm:UB} and Corollaries \ref{cor:poly}--\ref{cor:gauss}) sharpen these observations by precisely
quantifying the advantages of unsaturated regularization operators in terms of adaptability and
minimaxity.
In other related work, \citet{dhillon2013risk} have illustrated the potential advantages of KPCR
over KRR in finite-dimensional problems with linear kernels; though their work is not framed in
terms of saturation and general regularization operators, it relies on similar concepts.

We recently became aware of simultaneous independent work of \citet{blanchard2016optimal} that
proves upper bounds on the risk for a comparable class of regularization methods (which includes KRR
and KPCR), as well as minimax lower bounds that match the upper bounds; their results are
specialized to RKHS's where the covariance operator has polynomially decaying eigenvalues.
Compared to that work, our results require a weaker moment condition on the noise for risk bounds, apply to a much broader class of RKHS's, and also consider target functions that live in
finite-dimensional subspaces (see Proposition~\ref{prop:finite_adapt}).  Another
distinguishing feature of the present work is that it contains numerical
experiments that illustrate the implications of our theoretical
bounds in some practical settings (Section \ref{sec:experiments}).  

The main engine behind the technical results in this paper is a collection of large-deviation
results for Hilbert-Schmidt operators.
The required machinery is developed in the appendix.
These results build on straightforward extensions of results of \citet{tropp2015intro} and
\citet{minsker2011bernstein}.
Our most precise results for KPCR and achieving parametric rates for estimation over
finite-dimensional subspaces (Proposition~\ref{prop:finite_adapt}) rely on slightly different
arguments, which are based on well-known eigenvalue perturbation results that have been adapted to
handle Hilbert-Schmidt operators (e.g., the Davis-Kahan $\sin \Th$ theorem
\citep{davis1970rotation}).  

\section{Statistical setting and assumptions}

Our basic assumption on the distribution of $\z = (y,\x) \sim \rho$ is that the residual variance is
bounded; more specifically, we assume that there exists a constant $\s^2 > 0$ such that
\begin{equation}\label{resid}
  \int_{\cY} \left[ y - \reg(\x) \right]^2 \ d\rho(y\mid\x) \ \leq \ \s^2
\end{equation}
for almost all $\x \in \cX$.
\citet{zhang2013divide} also assume~\eqref{resid}; this assumption is slightly weaker than the
analogous assumption of \citet{bauer2007regularization} (Equation (1) in their paper).
Note that~\eqref{resid} holds if $y$ is bounded almost surely or if $y
= f^{\dagger}(\x) + \e$, where $\e$ is independent of $\x$, $\E(\e) = 0$, and $\Var(\e) = \s^2$.  

Let $K \colon \cX \times \cX \to \R$ be a symmetric positive-definite kernel function.
We assume that $K$ is bounded---i.e., that there exists $\k^2 > 0$
such that $\sup_{\x \in \cX} K(\x,\x) \ \leq \ \k^2$.
Additionally, we assume that there is a countable basis of eigenfunctions $\braces{\psi_j}_{j =
1}^{\infty} \subseteq L^2(\marginal)$ and a sequence of corresponding eigenvalues $t_1^2 \geq t_2^2
\geq \dotsb \geq 0$ such that
\begin{equation}\label{mercer}
  K(\x,\tilde\x) \ = \ \sum_{j = 1}^{\infty} t_j^2\psi_j(\x)\psi_j(\tilde\x) \,, \quad \x,\tilde\x \in \cX
\end{equation}
and the convergence is absolute.
Mercer's theorem and various generalizations give conditions under which representations
like~\eqref{mercer} are known to hold \citep{carmeli2006vector}; one of the simplest examples is
when $\cX$ is a compact Hausdorff space, $\marginal$ is a probability measure on the Borel sets of
$\cX$, and $K$ is continuous.
Observe that
\begin{align*}
  \sum_{j = 1}^{\infty} t_j^2
  & \ = \ \sum_{j = 1}^{\infty} t_j^2 \int_{\cX} \psi_j(\x)^2 \ \dif\marginal(\x)
  \ = \ \int_{\cX} K(\x,\x) \ \dif\marginal(\x)
  \ \leq \ \k^2 \,;
\end{align*}
in particular, $\braces{t_j^2} \in \ell^1(\N)$.  

Let $\H \subseteq L^2(\marginal)$ be the RKHS corresponding to
$K$  \citep{aronszajn1950theory} and let  $\phi_j \ = \ t_j\psi_j$, $j = 1,2,\ldots$.
It follows from basic facts about RKHSs that $\braces{\phi_j}_{j = 1}^{\infty}$ is an orthonormal
basis for $\H$ (if $t_J^2 > t_{J+1}^2 = 0$, then $\Braces{\phi_j}_{j=1}^J$ is an orthonormal basis for
$\H$).
Furthermore, $\H$ is characterized by
\[
  \H \ = \ \Braces{
    f = \sum_{j = 1}^{\infty} \th_j\psi_j \in L^2(\marginal);
    \ \sum_{j = 1}^{\infty} \frac{\th_j^2}{t_j^2} < \infty
  }
\]
and the inner product
\[
  \dotp{ f, \tilde f }_{\H}
  \ = \
  \Dotp{
    \sum_{j = 1}^{\infty} \th_j\psi_j,\sum_{j = 1}^{\infty} \tilde{\th}_j\psi_j
  }_{\H}
  \ = \
  \sum_{j = 1}^{\infty} \frac{\th_j\tilde{\th}_j}{t_j^2}
  \,
\]
(the corresponding norm is denoted by $\norm{\cdot}_{\H}$).
Our main assumption on the relationship between $y$, $\x$, and the kernel $K$ is that 
\begin{equation}\label{reg}
  \reg \ \in \ \H
  \,.
\end{equation}
This is a regularity or smoothness assumption on $\reg$. Many of the results in this paper can be modified, so that they apply to settings where $\reg \notin
\H$, by replacing $\reg$ with an appropriate projection of $\reg$ onto $\H$ and including an
approximation error term in the corresponding bounds.
This approach leads to the study of {\it oracle inequalities} \citep{zhang2005learning,zhang2013divide, 
koltchinskii2006local, steinwart2009optimal, hsu2014random}, which we do not pursue
in detail here.
However, investigating oracle inequalities for general regularization operators may be of interest
for future research, as most existing work focuses on ridge regularization.  

Another interpretation of condition \eqref{reg} is that it is a minimal regularity condition on
$\reg$ for ensuring that $\reg$ can be estimated consistently using the kernel methods considered
below.  One key aspect of the upper bounds in Section \ref{sec:main} is that they
show $\reg$ can be estimated more efficiently, if it satisfies
stronger regularity conditions (and if the regularization method used
is sufficiently adaptable). 
A convenient way to formulate a collection of regularity conditions
with varying strengths, which will be useful in the sequel, is as follows.
For $\zeta \geq 0$, define the Hilbert space
\begin{equation}\label{Hzeta}
  \H_\zeta
  \ = \
  \Braces{
    f = \sum_{j = 1}^{\infty} \th_j\psi_j \in L^2(\marginal);
    \ \sum_{j = 1}^{\infty} \frac{\th_j^2}{t_j^{2(1+\zeta)}} < \infty
  }
  \,.
\end{equation}
Then $\H = \H_0$ and $\H_{\zeta_2} \subseteq \H_{\zeta_1}$ whenever $\zeta_2 \geq \zeta_1 \geq 0$.
The norm on $\H_{\zeta}$ is defined by $\norm{ f }_{\H_{\zeta}}^2
  \ = \
  \sum_{j = 1}^{\infty} \frac{\th_j^2}{t_j^{2(1+\zeta)}}$.
Additionally, positive integers $J$ define the finite-rank subspace
\begin{equation}\label{HJ}
  \H_J^\circ
  \ = \
  \Braces{
    f = \sum_{j = 1}^{J} \th_j\psi_j \in L^2(\marginal);
    \ \th_1,\ldots,\th_J \in \R
  }
  \,.
\end{equation} 
We have the inclusion $\H_J^\circ \subseteq \H_{J+1}^\circ$ and, if $t_1^2,\ldots,t_J^2 > 0$, then
$\H_J^\circ \subseteq \H_\zeta$ for any $\zeta \geq 0$.
In particular, it is clear from \eqref{Hzeta}--\eqref{HJ} that $\reg \in \H_{\zeta}$ is a stronger regularity condition than $\reg \in \H$ and that $\reg \in \H_J^{\circ}$ is an even stronger condition (provided the $t_j^2$ are
strictly positive).
Conditions such as $\reg \in \H_{\zeta}$ and $\reg \in \H_J^{\circ}$ are known as {\em source
conditions} elsewhere in the literature \citep[e.g.,][]{bauer2007regularization,caponnetto2010cross}.

\section{Regularization} 

As discussed in Section~\ref{sec:intro}, our goal is find estimators $\est$ that minimize the
risk~\eqref{risk}.
In this paper, we focus on regularization-based estimators for $\reg$.
In order to precisely describe these estimators, we require some additional notation for various
operators that will be of interest, and some basic definitions from regularization theory.  

\subsection{Finite-rank operators of interest}

For $\x \in \cX$, define $K_{\x} \in \H$ by $K_{\x}(\tilde{\x}) = K(\x,\tilde{\x})$, $\tilde{\x} \in
\cX$.  Let $X = (\x_1,\dotsc,\x_n)^\top \in \R^{n \times d}$ and $\y = (y_1,\dotsc,y_n)^\top \in
\R^n$.
Additionally, define the finite-rank linear operators $S_X: \H \to \R^n$ and $T_X: \H \to \H$ (both
depending on $X$) by
\begin{align*}
  S_X\phi
  & \ = \
  (\dotp{ \phi,K_{\x_1} }_{\H}, \dotsc, \dotp{ \phi,K_{\x_n} }_{\H})^{\top}
  \ = \ (\phi(\x_1), \dotsc, \phi(\x_n))^{\top} \,, \\
  T_X\phi
  & \ = \
  \frac{1}{n}\sum_{i = 1}^n \dotp{ \phi,K_{\x_i} }_{\H} K_{\x_i}
  \ = \ \frac{1}{n} \sum_{i = 1}^n \phi(\x_i) K_{\x_i}
  \,,
\end{align*}
where $\phi \in \H$.
Let $\dotp{\cdot,\cdot}_{\R^n}$ denote the normalized inner-product on $\R^n$, defined by $\dotp{
\bfv,\tilde\bfv }_{\R^n} = n^{-1} \bfv^\top\tilde\bfv$ for $\bfv = (v_1,\dotsc,v_n)^\top, \
\tilde\bfv = (\tilde{v}_1,\dotsc,\tilde{v}_n)^\top \in \R^n$.
Then the adjoint of
$S_X$ with respect to $\dotp{\cdot,\cdot}_{\H}$ and
$\dotp{\cdot,\cdot}_{\R^n}$, $S_X^*: \R^n \to \H$, is given by 
$S_X^*\bfv
  \ = \
  \frac{1}{n}\sum_{i = 1}^n v_iK_{\x_i}$.  
Additionally, we have $T_X = S_X^*S_X$.
Finally, observe that $S_XS_X^* \colon \R^n \to \R^n$ is given by the $n \times n$ matrix $S_XS_X^*
= n^{-1}\K$, where $\K = (K(\x_i,\x_j))_{1 \leq i,j\leq n}$; $\K$ is the {\it kernel matrix}, which
is ubiquitous in kernel methods and enables finite computation.

\subsection{Basic definitions}

A family of functions $\linv \colon \intco{0,\infty} \to \intco{0,\infty}$ indexed by $\l > 0$ is
called a \emph{regularization family} if it satisfies the following three conditions:\footnote{%
  This definition follows \citet{engl1996regularization} and \citet{bauer2007regularization}, but is
  slightly more restrictive.%
}
\begin{description}
  \item[R1] $\sup_{0 < t \leq \k^2} |t \linv(t)| < 1$.
  \item[R2] $\sup_{0 < t \leq \k^2} |1 - t \linv(t)| \leq 1$. 
  \item[R3] $\sup_{0 < t \leq \k^2} | \linv(t)| < \l^{-1}$.
\end{description}
The main idea behind a regularization family is that it ``looks'' similar to $t \mapsto 1/t$, but is
better-behaved near $t = 0$, i.e., it is bounded by $\l^{-1}$.  
An important quantity that is related to the adaptability of a regularization
family is the {\em qualification} of the regularization.
The qualification of the regularization family $\braces{\linv}_{\l > 0}$ is defined to be the
maximal $\xi \geq 0$ such that
\[
  \sup_{0< t \leq \k^2}
  \abs{ 1 - t \linv(t) }
  t^{\xi}
  \ \leq \
  \l^{\xi}
  \,.
\]
If a regularization family has qualification $\xi$, we say that it ``saturates at $\xi$.''
Two regularization families that are the major motivation for the results in this paper are ridge
(Tikhonov) regularization, where $\linv(t) = \ \frac1{t+\l}$
and principal component (spectral cut-off) regularization, where
\begin{equation}\label{PCA}
  \linv(t) \ = \ \sinv(t) \ = \ \frac1t \ind{ t \geq \l }
  \,.
\end{equation}
Observe that ridge regularization has qualification 1 and principal
component regularization has qualification $\infty$.  
Another example of a regularization family with qualification $\infty$
is the Landweber iteration, which can be viewed as a special case of
gradient descent~\citep[see,
e.g.,][]{rosasco2005spectral,bauer2007regularization,logerfo2008spectral}.

\subsection{Estimators}
\label{sec:est}

Given a regularization family $\braces{\linv}_{\l > 0}$, we define the $\linv$-regularized estimator
for $\reg$, 
\begin{equation}\label{regest}
  \estl
  \ = \
  \linv(T_X) S_X^* \y
  \,.
\end{equation}
Here, $\linv$ acts on the spectrum (eigenvalues) of the finite-rank operator $T_X$ (which is the
same as the spectrum of the kernel matrix $\K$, up to scaling).
Therefore, a finitely-computable representation is $\estl = \sum_{i=1}^n \hat\gamma_i
K_{\x_i}$, where $(\hat\gamma_1,\dotsc,\hat\gamma_n)^\top = \linv(\K/n) \y / n$; computing the
$\hat\gamma_i$ involves an eigenvalue decomposition of the matrix $\K$.
The dependence of $\estl$ on the regularization family is implicit; our results hold for any
regularization family except where explicitly stated otherwise.
The estimators $\estl$ are the main focus of this paper.  

\section{Main results}\label{sec:main}

\subsection{General bound on the risk}

\begin{theorem}\label{thm:UB}
  Let $\estl$ be the estimator defined in \eqref{regest} with regularization family
  $\braces{\linv}_{\l>0}$.
  Let $0 \leq \d \leq 1$ and assume that there is some $\k_{\d}^2 > 0$ such that
  \begin{equation}\label{psibd}
    \sup_{\x \in \cX}
    \sum_{j = 1}^{\infty} t_j^{2(1 - \d)} \psi_j(\x)^2
    \ \leq \
    \k_{\d}^2
    \ < \
    \infty
  \end{equation}
  Assume that the source condition $\reg \in \H_{\zeta}$ holds for some $\zeta \geq 0$, and that $\linv$
  has qualification at least $\max\braces{(\zeta+1)/2,1}$.  
  Define the effective dimension $\dl =
    \sum_{j=1}^\infty
    \frac{t_j^2}{t_j^2 + \l}$.
  Finally, assume that $(8/3+2\sqrt{5/3}) \k_{\d}^2 / n \leq \l^{1-\d} \leq \k_{\d}^2$.
  The following risk bound holds:
  \begin{align}\label{UB}
    \risk(\estl)
    & \ \leq \
    2^{\zeta+3}\norm{\reg}_{\H_\zeta}^2 \l^{\zeta+1}
    + \frac{4\dl\s^2}{n}
    \\ \nonumber
    & \qquad
    + 4\dl \Parens{
      \norm{\reg}_{\H}^2 t_1^2
      + \frac{\k^2\s^2}{\l n}
    }
    \exp\Biggl(-\frac{3\l^{1-\d}n}{28\k_{\d}^2}\Biggr)
    \\ \nonumber
    & \qquad
    + \ind{\zeta>1} \cdot 16\zeta^2(3/2)^{\zeta-1}
    \cdot \norm{\reg}_{\H_\zeta}^2
    (t_1^2+\l)^{\zeta-1} \k^4
    \Biggl(
      \frac{34}{n} + \frac{15}{n^2}
    \Biggr)
    \,.
  \end{align}
\end{theorem}

Theorem \ref{thm:UB} is proved in Appendix~\ref{sec:proof-UB}.  The
first two terms in the upper bound \eqref{UB} are typically the dominant terms.  
In the upper bound \eqref{UB}, the interaction between the kernel $K$ and the distribution $\marginal$ is reflected in the
\emph{effective dimension} $\dl$~\citep[see, e.g.,][]{zhang2005learning,caponnetto2007optimal}.
The regularity of $\reg$ enters through norm of $\reg$ (both the $\H$-
and $\H_{\zeta}$-norms) and the
exponent on $\l$.

The
condition \eqref{psibd} in Theorem \ref{thm:UB} is always satisfied by taking $\d = 0$ and
$\kappa_0^2 = \kappa^2$.  Requiring \eqref{psibd} with $\d > 0$
imposes additional conditions on the RKHS $\H$.  For Corollary
\ref{cor:poly} below (which applies when the eigenvalues $\{t_j^2\}$
have polynomial-decay), we take $\d = 0$ and
$\kappa_0^2 = \kappa^2$.  The stronger condition with $\d > 0$ is required
to obtain obtain minimax rates for
kernels where the eigenvalues $\{t_j^2\}$ have exponential or
Gaussian-type decay (see Corollaries \ref{cor:exp}--\ref{cor:gauss}).

Risk bounds on general regularization
estimators similar to $\estl$ were previously obtained by
\citet{bauer2007regularization}.
However, their bounds \citep[e.g., Theorem 10 in][]{bauer2007regularization} are independent of the
ambient RKHS $\H$, i.e., they do not depend on the eigenvalues $\braces{t_j^2}$.
Our bounds are tighter than those of \citet{bauer2007regularization} because we take advantage of
the structure of $\H$.  
In contrast with our Theorem \ref{thm:UB}, the results of \citet{bauer2007regularization} do not give
minimax bounds (not easily, at least), because minimax rates must
depend on the $t_j^2$.  

\subsection{Implications for kernels characterized by their eigenvalues' rate of
  decay}  

We now state consequences of Theorem~\ref{thm:UB} that give explicit rates for estimating $\reg$ via
$\estl$, for any regularization family, under specific assumptions about the decay rate of the
eigenvalues $\{t_j^2\}$.

We first consider the case where the eigenvalues have polynomial decay.

\begin{corollary}\label{cor:poly}
  Assume that $C \geq 0$ and $\nu > 1/2$ are constants such that $0 < t_j^2 \leq Cj^{-2\nu}$ for all
  $j = 1,2,\ldots$.
  Assume the source condition $\reg \in \H_{\zeta}$ for some $\zeta \geq 0$, and that $\linv$
  has qualification at least $\max\braces{(\zeta+1)/2,1}$.
  Finally, take $\l = C' n^{-\frac{2\nu}{2\nu(\zeta + 1) + 1}}$
  for a suitable constant $C'>0$ so that the conditions on $\l$ from Theorem~\ref{thm:UB} are satisfied.
  Then
  \begin{align*}
    \risk(\estl)
    & \ = \
    O\Parens{
      \Braces{ \norm{\reg}_{\H_\zeta}^2 + \s^2 }
      \cdot
      n^{-\frac{2\nu(\zeta+1)}{2\nu(\zeta + 1) + 1}}
    }
    \,,
  \end{align*}
  where the constants implicit in the big-$O$ may depend on $\k^2$, $C$, $C'$, $\nu$, and
  $\zeta$, but nothing else.
\end{corollary}

\begin{remark}
  Observe that if $\linv$ has qualification at least $\max\braces{(\zeta+1)/2,1}$ (and the other
  conditions of Corollary \ref{cor:poly} are met), then $\hat{f}_{\l}$ obtains the minimax rate for
  estimating functions over $\H_{\zeta}$ \citep{pinsker1980optimal}.
  Thus, if $g_{\lambda}$ has higher qualification, then $\hat{f}_{\l}$
  can effectively adapt to a broader range of subspaces 
  $\H_{\zeta} \subseteq \H = \H_0$.   In particular, KPCR (with infinite qualification) can
 adapt to source conditions with arbitrary $\zeta \geq
  0$; on the other hand, KRR satisfies the conditions of Corollary \ref{cor:poly} only when $\zeta
  \leq 1$, because KRR has qualification 1.  
\end{remark}

\begin{remark}
  As mentioned earlier, a very similar result for polynomial decay eigenvalues was independently and
  simultaneously obtained by \citet{blanchard2016optimal} for essentially the same class of
  regularization operators.  Our Theorem \ref{thm:UB}, from
  which the corollary follows, applies to a broader
  class of kernels than results of \citeauthor{blanchard2016optimal}.
\end{remark}

When the eigenvalues $\braces{t_j^2}$ have exponential or Gaussian-type decay, the rates are nearly
the same as in finite dimensions.

\begin{corollary}\label{cor:exp}
  Assume that $C, \a \geq 0$ are constants such that $0 < t_j^2 \leq Ce^{-\a j}$ for all
  $j = 1,2,\ldots$.
  Assume that $\linv$
  has qualification at least $1$ and that \eqref{psibd} holds for any $0 < \d \leq 1$.
  Finally, take $\l = C'n^{-1}\log(n)$ for a suitable constant $C'>0$ so that the conditions on $\l$
  from Theorem~\ref{thm:UB} are satisfied.
  Then
  \begin{align*}
    \risk(\estl)
    & \ = \
    O\Parens{
      \Braces{ \norm{\reg}_{\H}^2 + \s^2 }
      \cdot
      \frac{\log(n)}{n}
    }
    \,,
  \end{align*}
  where the constants implicit in the big-$O$ may depend on $\k^2$, $C$, $C'$, $\a$, $\d$,
  and $\k_{\d}^2$, but nothing else.  
\end{corollary}

\begin{corollary}\label{cor:gauss}
  Assume that $C, \a \geq 0$ are constants such that $0 < t_j^2 \leq Ce^{-\a j^2}$ for all
  $j = 1,2,\ldots$.
  Assume that $\linv$
  has qualification at least $1$ and that \eqref{psibd} holds for any $0 < \d \leq 1$.
  Finally, take $\l = C'n^{-1}\sqrt{\log(n)}$ for a suitable constant $C'>0$ so that the conditions
  on $\l$ from Theorem~\ref{thm:UB} are satisfied.
  Then
  \begin{align*}
    \risk(\estl)
    & \ = \
    O\Parens{
      \Braces{ \norm{\reg}_{\H}^2 + \s^2 }
      \cdot
      \frac{\sqrt{\log(n)}}{n}
    }
    \,,
  \end{align*}
  where the constants implicit in the big-$O$ may depend on $\k^2$, $C$, $C'$, $\a$, $\d$,
  and $\k_{\d}^2$, but nothing else.  
\end{corollary}

\begin{remark}
  In Corollaries \ref{cor:exp} --\ref{cor:gauss}, we get minimax
  estimation over $\H = \H_0$ \citep{pinsker1980optimal}.
  However, our bounds are not refined enough to pick up any potential improvements which may be had
  if  a stronger source condition is satisfied (e.g., $\reg \in \H_{\zeta}$ for $\zeta > 0$) .
  This is typical in settings like this because the minimax rate is already quite fast, i.e., within
  a log-factor of the parametric rate $n^{-1}$.  
\end{remark}

\subsection{Parametric rates for finite-dimensional kernels and subspaces}

\label{sec:finite}

If the kernel has finite rank (i.e., $t_j^2 = 0$ for $j$ sufficiently large), then it follows directly from Theorem
\ref{thm:UB} that 
$\risk(\estl) = O\parens{
    \braces{ \norm{\reg}_{\H}^2 + \s^2 } / n
  }$.
If the kernel has infinite rank, but $\reg $ is contained in the finite-dimensional subspace
$\H_J^{\circ}$ for some
$J < \infty$, then Theorem \ref{thm:UB}
can still be applied, provided $\linv$ has high qualification.  Indeed, if $\linv$
has infinite qualification and $\reg \in
\H_J^{\circ}$, then it follows that $\reg \in \H_{\zeta}$ for all $\zeta
\geq 0$ and Theorem \ref{thm:UB} implies that for any $0 < \a \leq 1$, 
  $\risk(\estl)
  =
  O\parens{
    \braces{ \norm{\reg}_{\H}^2 + \s^2 } / n^{1-\a}
  }$ for appropriately chosen $\l$.  
In fact, we can improve on this rate for KPCR.  
The next proposition implies that the risk of KPCR matches the
parametric rate $n^{-1}$ for $\reg \in \H_J^{\circ}$; the proof requires a different 
 argument, based on eigenvalue perturbation theory, which we give in
 Appendix~\ref{sec:proof-finite_adapt}.

\begin{proposition}\label{prop:finite_adapt}
  Let $\kpcr$ be the KPCR estimator, with principal component regularization
  \eqref{PCA}, and assume that $\reg \in \H_J^{\circ}$.
  Let $0 < r < 1$ be a constant and let $\l = (1-r)t_J^2$.
  If $rt_J^2 \geq \k^2/n^{1/2} + \k^2/(3n)$, then
  \begin{align*}
    \risk(\kpcr)
    & \ \leq \
    \frac{1}{n}\cdot
    \Braces{
      \frac{34\kappa^6}{r^2t_J^4}
      \norm{\reg}^2_{\H} + \frac{3\k^2}{(1-r)t_J^2}\s^2
    }
    + \k^2 \norm{\reg}{\H}^2
    \Braces{
      \frac{15\k^4}{n^2r^2t_J^4}
      + 4\exp\Parens{
        -\frac{nr^2t_J^4}{2\k^4 + 2\k^2rt_J^2/3}
      }
    }
    \,.
  \end{align*}
\end{proposition}

Proposition~\ref{prop:finite_adapt} implies that KPCR may reach the parametric rate for estimating $\reg
\in\H_J^{\circ}$.  On the other hand, it is known that KRR may perform dramatically worse than KPCR in these
settings due
to the saturation effect \citep[see, e.g.,][]{caponnetto2007optimal,
  dicker2016ridge, dhillon2013risk}.

\section{Numerical experiments}\label{sec:experiments} 

\subsection{Simulated data}

This simulation study shows how KPCR is able to adapt to
highly structured signals, while KRR requires more favorable structure
from the ambient RKHS.  

For this experiment, we take $\cX = \{1,2,\ldots,2^{13}\}$.  The data distribution
$\rho$ on $\cY \times \cX$ is specified as follows.
The marginal distribution on $\cX$ is $\marginal(x) \propto x^{-1/2}$, the
regression function $\reg$ is given by $\reg(x) = \sum_{j=1}^5 \ind{x=j}$, and $\rho(\cdot\mid x)$
is normal with mean $\reg(x)$ and variance $1/4$.  To compute
$\hat{f}_{\l}$, we use the discrete kernel $K(x,\tilde x) = \ind{x = \tilde x}$

Using an iid~sample of size $n=2^{13}$, we compute $\estl$ (either KRR or KPCR) for $\l$ in a discrete
grid of $2^{10}$ values uniformly spaced between $10^{-5}$ and $0.02$, and then choose the value of
$\l$ for which $\estl$ has smallest \emph{validation mean-squared error}
$n^{-1} \sum_{i=1}^n (y_i^v-\estl(x_i^v))^2$,
computed using a separate iid~sample $\braces{ (x_i^v,y_i^v)
}_{i=1}^n$ of size $n=2^{13}$.  

\begin{figure}
  \begin{center}
    \begin{tabular}{cc}
      \includegraphics[width=0.47\textwidth]{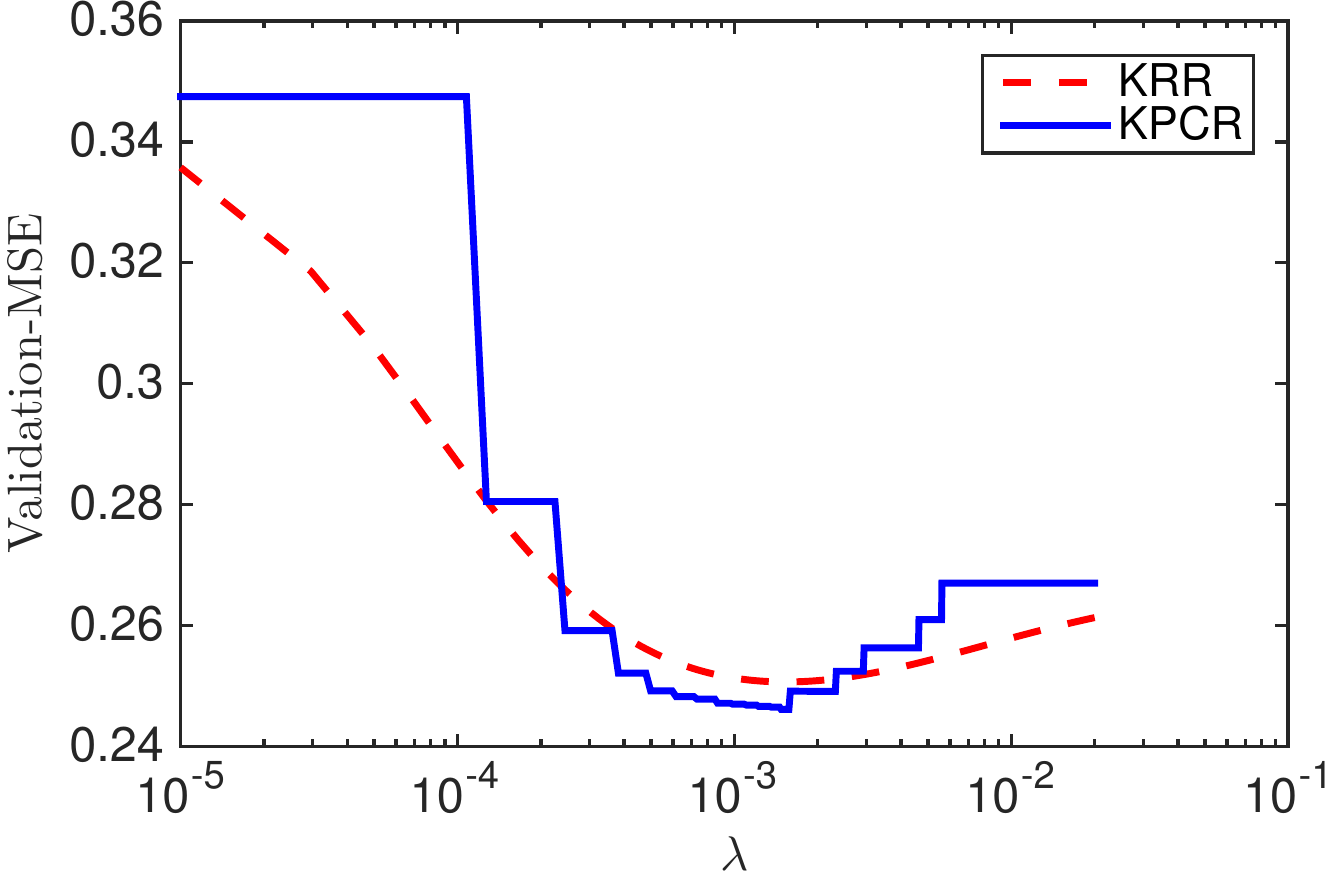} &
      \includegraphics[width=0.47\textwidth]{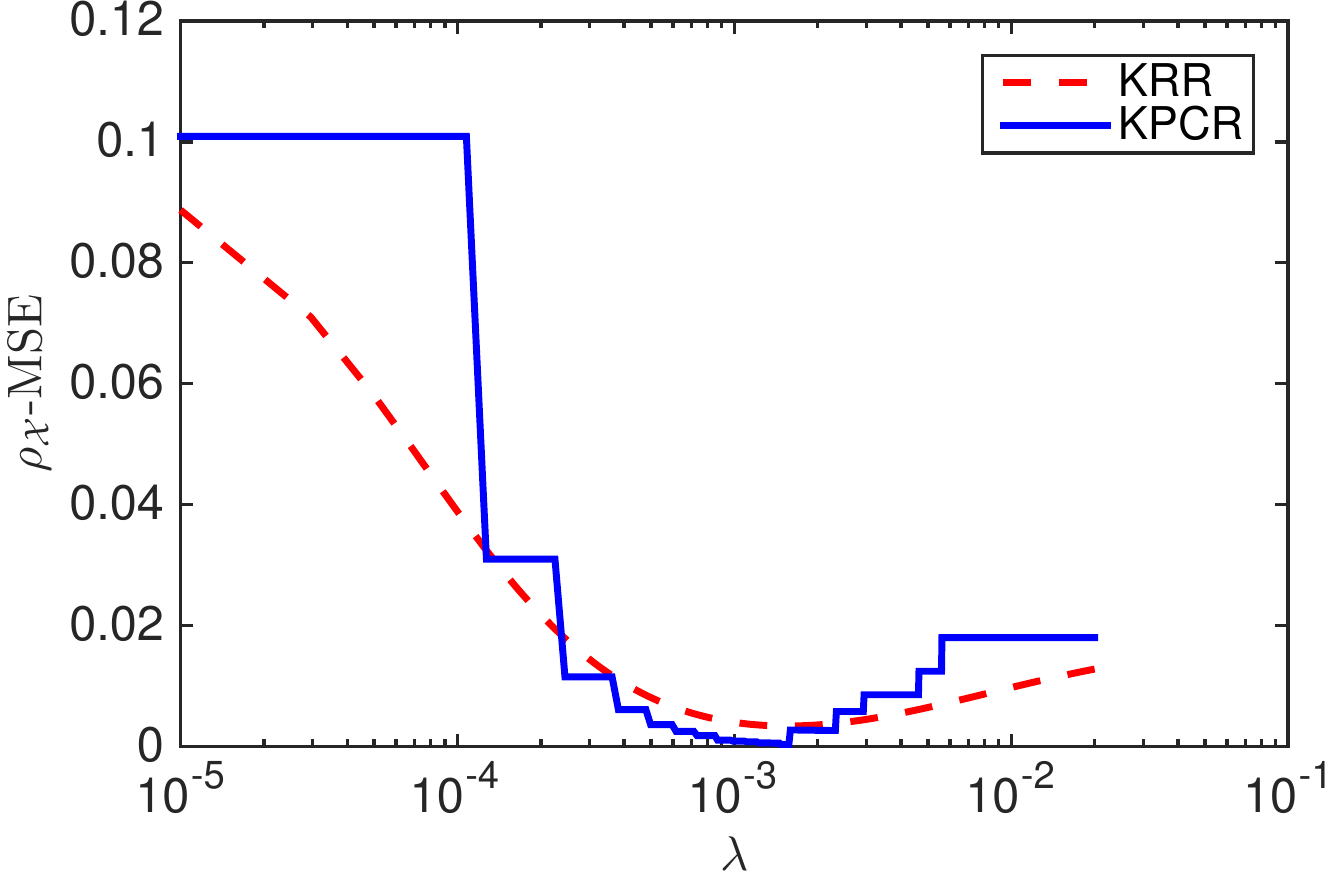}
      \\
      (a) & (b)
      \\
      \includegraphics[width=0.49\textwidth]{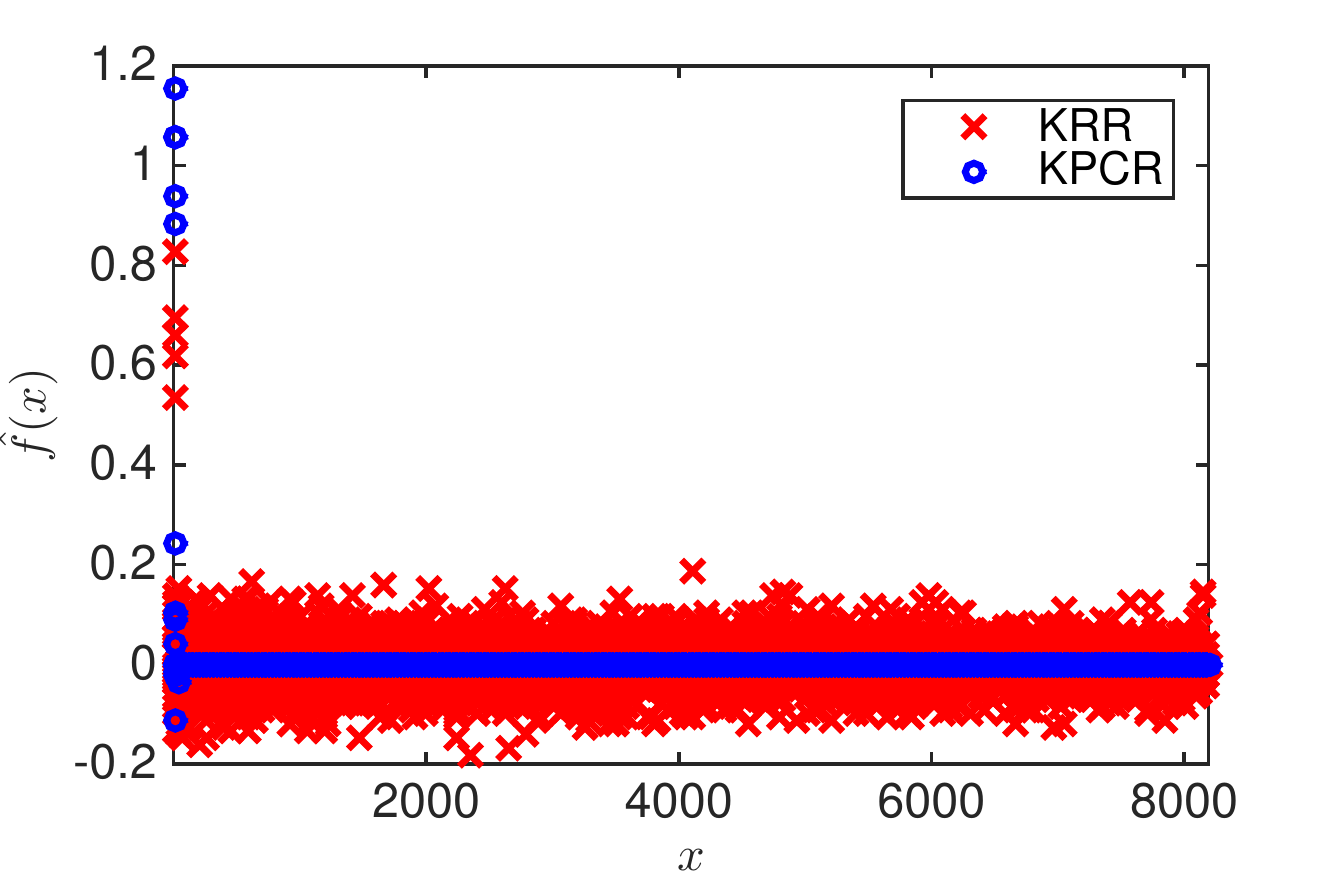} &
      \includegraphics[width=0.47\textwidth]{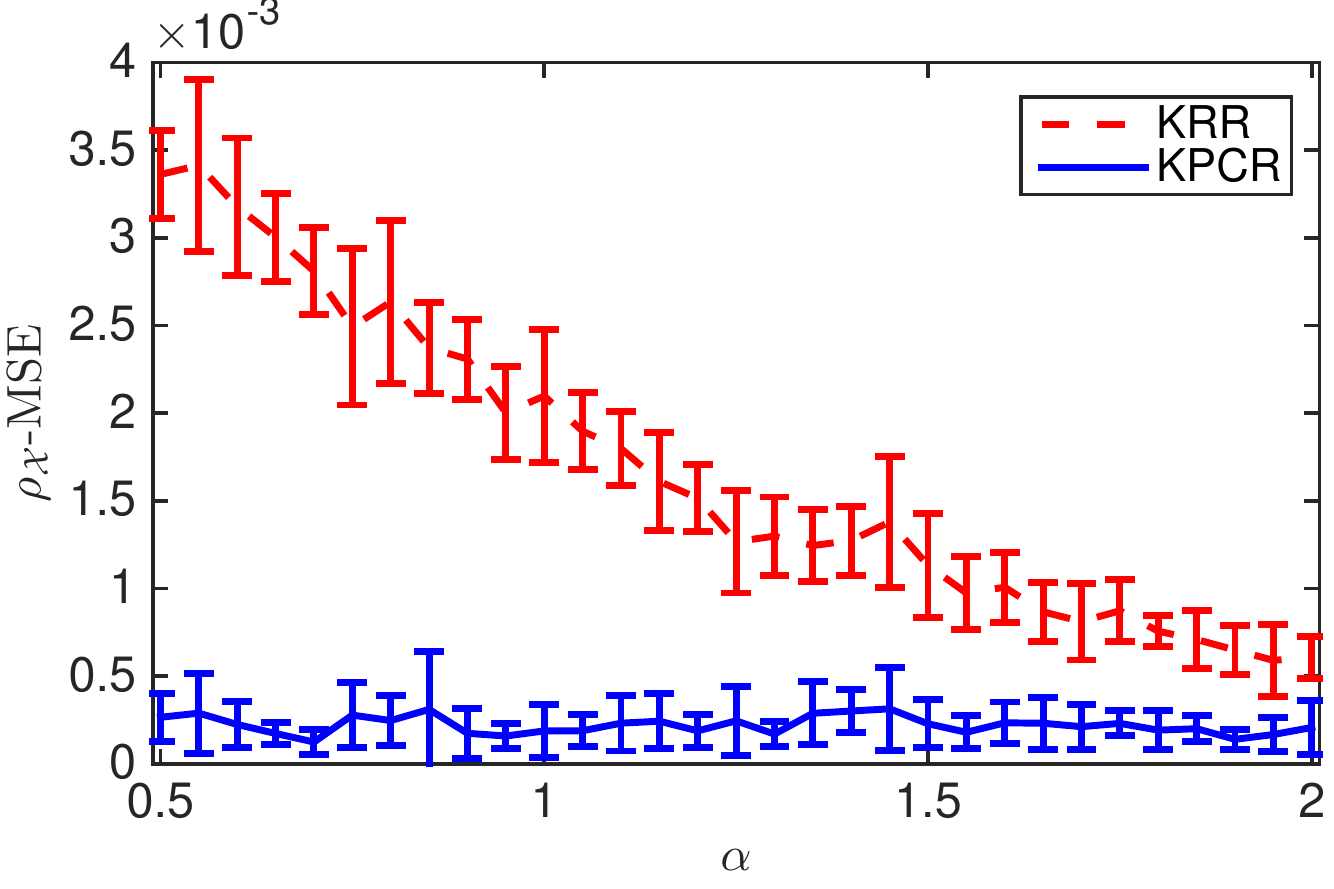}
      \\
      (c) & (d)
    \end{tabular}
  \end{center}
  \caption{\label{fig:mse}
    (a) Validation-MSE of $\krr$ and $\kpcr$
    for $\marginal(x) \propto x^{-1/2}$ as $\l$ varies;
    (b) $\marginal$-MSE of $\krr$ and $\kpcr$
    for $\marginal(x) \propto x^{-1/2}$ as $\l$ varies;
    (c) estimated functions $\est_{\KRR,\hat\l_{\KRR}}$ and
    $\est_{\KPCR,\hat\l_{\KPCR}}$ for $\marginal(x) \propto x^{-1/2}$;
    (d) $\marginal$-MSE of $\est_{\KRR,\hat\l_{\KRR}}$ and
    $\est_{\KPCR,\hat\l_{\KPCR}}$ for $\marginal(x) \propto x^{-\a}$ as $\a$
    varies.
  }
\end{figure}

Figure~\ref{fig:mse}(a) shows the validation-MSE of each $\estl$; the plot of \emph{$\marginal$-MSE}
$\norm{\reg-\estl}_{L^2(\marginal)}^2$ has roughly the same shape, just shifted down by $\s^2=1/4$
(Figure~\ref{fig:mse}(b)).
The selected $\l$ is $\hat\l_{\KRR} = 0.001534$ for KRR, and $\hat\l_{\KPCR} = 0.001573$ for KPCR.
These choices of $\l$ yield the final estimators, $\est_{\KRR,\hat\l_{\KRR}}$ and
$\est_{\KPCR,\hat\l_{\KPCR}}$; the $\marginal$-MSE is $0.0034$ for KRR, and $0.0003$ for KPCR.
In Figure~\ref{fig:mse}(c),
we plot the functions $\est_{\KRR,\hat\l_{\KRR}}$ and
$\est_{\KPCR,\hat\l_{\KPCR}}$; the KRR function is non-zero for much of the domain, while the KPCR
function is zero for nearly all of the domain (like $\reg$).

We repeat the above simulation for different marginal distributions $\marginal(x) \propto x^{-\a}$,
for $1/2 \leq \a \leq 2$, which imply different eigenvalue sequences $\{t_j^2\}$.
The mean and standard deviation of the $\marginal$-MSE's over $10$ repetitions are shown in
Figure~\ref{fig:mse}(b).
This confirms KPCR's to adapt to the regularity of $\reg$, regardless of the ambient RKHS;
KRR requires more structure to achieve similar results.

\subsection{Real data}

We also compared KRR and KPCR using three ``weighted degree'' kernels designed for recognizing
splice sites in genetic sequences~\citep{sonnenburg08machine}.\footnote{We use the first three
kernels for the data obtained from \url{http://mldata.org/repository/data/viewslug/mkl-splice/}.}
The $3300$ samples are divided into a training set ($1000$), validation set ($1100$), and testing set
($1200$).
For each kernel, we use the training data to compute $\estl$ for $\l$ in a discrete grid of $2^{10}$
equally-spaced values between $10^{-5}$ and $0.4$, and select the value of $\l$ on which the MSE of
$\estl$ on the validation set is smallest.
The MSE on the testing set and the intrinsic dimension $d_{\hat\lambda}$ for the selected
$\hat\lambda$ (on the training data) are as follows:
\begin{center}
  \begin{tabular}{|c|c|c|c|}
    \hline
    & Kernel 1 & Kernel 2 & Kernel 3 \\
    \hline
    MSE of $\est_{\KRR,\hat\l_{\KRR}}$
    & \textbf{0.181452}
    & 0.172223
    & 0.167745 \\
    \hline
    MSE of $\est_{\KPCR,\hat\l_{\KPCR}}$
    & 0.187059
    & \textbf{0.168067}
    & \textbf{0.164159} \\
    \hline
    $d_{\hat\l_{\KRR}}$
    & 175.6879
    & 373.0029
    & 738.0712 \\
    \hline
    $d_{\hat\l_{\KPCR}}$
    & 170.9016
    & 275.3560
    & 581.1381 \\
    \hline
  \end{tabular}
\end{center}
KRR outperforms KPCR with Kernel 1, where the intrinsic dimension of the kernel is low, while the
reverse happens with Kernels 2 and 3, where the intrinsic dimensions
are high.  This resonates with our theoretical results, which suggest that
KRR requires low intrinsic  dimension to perform most effectively.  

\section{Discussion}

Our unified analysis for a general class of regularization families in
nonparametric regression highlights two important statistical
properties.
First, the results show minimax optimality for this general class in
several commonly studied settings, which was only previously
established for specific regularization methods.
Second, the results demonstrate the adaptivity of certain
regularization families to subspaces of the RKHS, showing that
these techniques may take advantage of additional smoothness
properties that the signal may possess.
It is notable that the most well-studied family, KRR/Tikhonov
regularization, does not possess this adaptability property.

\bibliographystyle{plainnat}
\bibliography{paper}

\newpage

\appendix

\section{Proof of Theorem~\ref{thm:UB}}
\label{sec:proof-UB}

To provide some intuition behind $\estl$ and our proof strategy, define the positive self-adjoint
operator
\begin{equation*}
  T \colon \H \to \H
\end{equation*}
by
\begin{align*}
  T\phi
  & \ = \
  \int_{\cX} \dotp{ \phi,K_{\x} }_{\H} K_{\x} \ \dif\marginal(\x)
  \ = \
  \sum_{j=1}^{\infty} t_j^2 \dotp{ \phi,\phi_j }_{\H}\phi_j
  \,, \quad \phi \in \H
  \,.
\end{align*}
Observe that $T$ is a ``population'' version of the operator $T_X$.
Unlike $T_X$, $T$ often has infinite rank; however, we still might expect that 
\begin{equation*}
  T
  \ \approx \
  T_X
\end{equation*}
for large $n$, where the approximation holds in some suitable sense.

We also have a large-$n$ approximation for $S_X^*\y$.
For $\phi \in \H$,
\begin{align*}
  \dotp{ \phi,S_X^*\y }_{\H}
  & \ = \ \frac{1}{n} \sum_{i = 1}^n y_i \dotp{ \phi,K_{\x_i} }_{\H}
  \\
  & \ \approx \ \int_{\cY \times \cX} y \phi(\x) d\rho(y,\x)
  \ = \ \int_{\cX} \reg(\x)\phi(\x) \dif\marginal(\x)
  \ = \ \dotp{ \phi,\reg }_{L^2(\marginal)}
  \ = \ \dotp{ \phi,T\reg }_{\H}
  \,,
\end{align*}
where $\dotp{ \cdot,\cdot }_{L^2(\marginal)}$ denotes the inner-product on $L^2(\marginal)$ and we
have used the identity $\phi_j = t_j\psi_j$ to obtain the last equality.
It follows that $S_X^*\y \approx T\reg$.
Hence, to recover $\reg$ from $\y$, it would be natural to apply the inverse of $T$ to $S_X^*\y$.
However, $T$ is not invertible whenever it has infinite rank, and regularization becomes necessary.
We thus arrive at the chain of approximations which help motivate $\estl$:
\[
  \estl
  \ = \
  \linv(T_X)S_X^*\y
  \ \approx \ \linv(T)T\reg
  \ \approx \ \reg
  \,,
\]
where $\linv(T)$ may be viewed as an approximate inverse for a suitably chosen regularization
parameter $\l$.

\subsection{Bias-variance decomposition}

The proof of Theorem~\ref{thm:UB} is based on a simple bias-variance
decomposition of the risk of $\estl$.  Let $\ee = (y_1-\reg(\x_1),\dotsc,y_n-\reg(\x_n))^\top \in \R^n$.
\begin{proposition}\label{prop:bv}
 The risk $\risk(\estl)$ has the decomposition
  \begin{equation}
    \risk(\estl)
    \ = \
    \bias(\estl) + \variance(\estl)
    \,,
  \end{equation}
  where
  $\bias(\estl)$ and $\variance(\estl)$ are defined as
  \begin{align*}
    \bias(\estl)
    & \ = \ \E\Brackets{ \Norm{ T^{1/2}\braces{ I - \linv(T_X)T_X
    }\reg }_{\H}^2 } \,, &
    \variance(\estl)
    & \ = \
    \E\Brackets{ \Norm{ T^{1/2}\linv(T_X)S_X^*\ee }_{\H}^2 }
    \,.
  \end{align*}
\end{proposition}

Our proof separately bounds the bias $\bias(\estl)$ and variance $\variance(\estl)$ terms from
Proposition~\ref{prop:bv}.
Taken together, these bounds imply a bound on $\risk(\estl)$.

\subsection{Translation to vector and matrix notation}
\label{sec:ell2-notations}

We first note that the Hilbert space $\H$ is isometric to $\ell^2(\N)$ via the isometric isomorphism
$\iota \colon \H \to \ell^2(\N)$, given by
\begin{equation}\label{isometry}
  \iota \colon \sum_{j= 1}^{\infty} \a_j \phi_j
  \ \mapsto \
  (\a_1,\a_2,\dotsc)^{\top}
\end{equation}
(we take all elements of $\ell^2(\N)$ to be infinite-dimensional column vectors).
Using this equivalence, we can convert elements of $\H$ and linear operators on $\H$ appearing in
Proposition~\ref{prop:bv} into (infinite-dimensional) vectors and matrices, respectively, which we
find simpler to analyze in the sequel. 

Define the (infinite-dimensional) diagonal matrix
\[
  \T \ = \ \diag(t_1^2,t_2^2,\dotsc)
  \,,
\]
and the vector
\[
  \bb \ = \ (\b_1,\b_2,\dotsc)^{\top}
  \ \in \ \ell^2(\N)
  \,,
\]
where $\b_i = \dotp{ \reg, \phi_i }_{\H}$.
Next, define the $n \times \infty$ (random) matrices
\begin{align*}
  \bPsi
  & \ = \ (\psi_j(\x_i))_{1 \leq i \leq n; 1 \leq j < \infty}
  \,,
  \\
  \bPhi
  & \ = \ \bPsi\T \ = \ (\phi_j(\x_i))_{1 \leq i \leq n; 1 \leq j < \infty}
  \,.
\end{align*}
Observe that
\begin{align*}
  S_X & \ = \ \bPhi \circ \iota \,,
  \\
  S_X^* & \ = \ \iota^{-1} \circ \Parens{\frac{1}{n}\bPhi^{\top}} \,,
  \\
  T & \ = \ \iota^{-1} \circ \T \circ \iota \,.
\end{align*}
Also, for $1 \leq i \leq n$, let $\bphi_i\bphi_i^\top$ be the $\infty \times \infty$ matrix whose
$(j,j')$-th entry is $\phi_j(\x_i)\phi_{j'}(\x_i)$, and define
\begin{align*}
  \S
  & \ = \
  \frac1n \sum_{i=1}^n \bphi_i\bphi_i^\top
  \ = \
  \frac1n \bPhi^{\top}\bPhi
  \,.
\end{align*}
Finally, let $\I = \diag(1,1,\dotsc)$, and let $\norm{\cdot} = \norm{\cdot}_{\ell^2(\N)}$ denote the
norm on $\ell^2(\N)$.  
In these matrix and vector notations, the bias-variance decomposition from Proposition~\ref{prop:bv}
translates to the following:
\begin{align*}
  \bias(\estl) & \ = \ \E\Brackets{ \Norm{ \T^{1/2}\braces{\I - \linv(\S)\S}\bb }^2 }
  \,, \\
  \variance(\estl) & \ = \ \frac{1}{n^2}\E\Brackets{ \Norm{ \T^{1/2}\linv(\S)\bPhi^{\top}\ee }^2 }
  \,.
\end{align*}
The boundedness of the kernel implies
\begin{align*}
  \tr(\T) & \ \leq \ \k^2 \,, \\
  \norm{\bphi_i\bphi_i^\top} & \ \leq \ \k^2 \,, \\
  \norm{\S} & \ \leq \ \k^2 \,,
\end{align*}
where the norms are the operator norms in $\ell^2(\N)$.

\subsection{Probabilistic bounds}

For $0 < r < 1$, define the event
\begin{equation}
  \cA_r
  \ = \
  \Biggl\{
    \Norm{
      (\T+\l\I)^{-1/2} (\S-\T) (\T+\l\I)^{-1/2}
    }
    \ \geq \
    r
  \Biggr\}
  \,,
\end{equation}
and let $\cA_r^c$ denote its complement.
Our bounds on bias $\bias(\estl)$ and variance $\variance(\estl)$ are based on analysis in the event
$\cA_r^c$ (for a constant $r$).
Therefore, we also need to show that $\cA_r^c$ has large probability (equivalently, show that
$\cA_r$ has small probability).

\begin{lemma}
  \label{lem:tail}
  Assume that
  \begin{align}
    \sup_{\x \in \cX}
    \sum_{j=1}^\infty t_j^{2(1-\d)} \psi_j(\x)^2
    \ \leq \
    \k_{\d}^2
    \ < \
    \infty
  \end{align}
  for some $0 \leq \d < 1$, $\k_{\d}^2 >0$.
  Further assume that $\l^{1-\d} \leq \k_{\d}^2$.
  If $r \geq \sqrt{\k_{\d}^2/(\l^{1-\d}n)} + \k_{\d}^2/(3\l^{1-\d}n)$, then
  \begin{align*}
    P(\cA_r)
    & \ \leq \
    4\dl
    \exp\Parens{
      -\frac{\l^{1-\d}nr^2}{2\k_{\d}^2(1+r/3)}
    }
    \,.
  \end{align*}
\end{lemma}
\begin{proof}
  The proof is an application of Lemma~\ref{lem:minsker}.
  Define, for $1 \leq i \leq n$,
  \[
    \X_i
    \ = \
    \frac1n
    \parens{ \T + \l\I }^{-1/2}
    \braces{
      \bphi_i\bphi_i^\top - \T
    }
    \parens{ \T + \l\I }^{-1/2}
    \,,
  \]
  as well as $\Y = \sum_{i=1}^n \X_i$.
  We have
  \[
    \Y \ = \ (\T+\l\I)^{-1/2} (\S-\T) (\T+\l\I)^{-1/2}
    \,.
  \]
  It is clear that $\E\brackets{\X_i} = \mathbf{0}$.
  Observe that
  \begin{align}
    \sum_{j=1}^\infty \frac{t_j^2}{t_j^2+\l} \psi_j(\x_i)^2
    & \ \leq \
    \max_{j\geq1} \frac{t_j^{2\d}}{t_j^2+\l}
    \sum_{j=1}^\infty t_j^{2(1-\d)} \psi_j(\x_i)^2
    \ \leq \
    \k_{\d}^2
    \max_{j\geq1} \frac{t_j^{2\d}}{t_j^2+\l}
    \ \leq \
    \frac{\k_{\d}^2}{\l^{1-\d}}
    \,,
    \label{eq:length}
  \end{align}
  where the last inequality uses the inequality of arithmetic and geometric means.
  Therefore, by the assumption $\l^{1-\d} \leq \k_{\d}^2$,
  \begin{align*}
    \norm{\X_i}
    & \ \leq \
    \frac1n
    \max\Braces{
      \norm{
        \parens{ \T + \l\I }^{-1/2} \bphi_i\bphi_i^\top \parens{ \T + \l\I }^{-1/2}
      }
      \,,\,
      \norm{
        \parens{ \T + \l\I }^{-1/2} \T \parens{ \T + \l\I }^{-1/2}
      }
    }
    \\
    & \ = \
    \frac1n
    \max\Braces{
      \sum_{j=1}^\infty \frac{t_j^2}{t_j^2+\l} \psi_j(\x_i)^2
      \,,\,
      \max_{j\geq1} \frac{t_j^2}{t_j^2+\l}
    }
    \ \leq \
    \frac{\k_{\d}^2}{\l^{1-\d}n}
    \,.
  \end{align*}
  Moreover,
  \begin{align*}
    \E\brackets{\X_i^2}
    & \ = \
    \frac1{n^2}
    \E\Brackets{
      \parens{ \T + \l\I }^{-1/2}
      \bphi_i\bphi_i^\top
      \parens{ \T + \l\I }^{-1}
      \bphi_i\bphi_i^\top
      \parens{ \T + \l\I }^{-1/2}
      - \parens{ \T + \l\I }^{-2} \T^2
    }
    \\
    & \ = \
    \frac1{n^2}
    \E\Brackets{
      \Braces{
        \sum_{j=1}^\infty \frac{t_j^2}{t_j^2+\l} \psi_j(\x_i)^2
      }
      \parens{ \T + \l\I }^{-1/2}
      \bphi_i\bphi_i^\top
      \parens{ \T + \l\I }^{-1/2}
      - \parens{ \T + \l\I }^{-2} \T^2
    }
    \,.
  \end{align*}
  Combining this with~\eqref{eq:length} gives
  \begin{align*}
    \norm{\E\brackets{\Y^2}}
    & \ \leq \
    \frac{\k_{\d}^2}{\l^{1-\d}n}
    \Norm{
      \E\Brackets{
        \parens{ \T + \l\I }^{-1/2}
        \bphi_1\bphi_1^\top
        \parens{ \T + \l\I }^{-1/2}
      }
    }
    \\
    & \ = \
    \frac{\k_{\d}^2}{\l^{1-\d}n}
    \norm{ \parens{ \T + \l\I }^{-1} \T }
    \ \leq \
    \frac{\k_{\d}^2}{\l^{1-\d}n}
  \end{align*}
  and
  \begin{align*}
    \tr\Parens{\E\brackets{\Y^2}}
    & \ \leq \
    \frac{\k_{\d}^2}{\l^{1-\d}n}
    \tr\Parens{
      \E\Brackets{
        \parens{ \T + \l\I }^{-1/2}
        \bphi_1\bphi_1^\top
        \parens{ \T + \l\I }^{-1/2}
      }
    }
    \\
    & \ = \
    \frac{\k_{\d}^2}{\l^{1-\d}n}
    \tr\Parens{ \parens{ \T + \l\I }^{-1} \T }
    \ = \
    \frac{\dl\k_{\d}^2}{\l^{1-\d}n}
    \,.
  \end{align*}
  Applying Lemma~\ref{lem:minsker} with $V = R = \k_{\d}^2 / (\l^{1-\d}n)$ and
  $D = \dl$, gives
  \[
    P\Parens{ \norm{\Y} \geq r }
    \ \leq \
    4\dl\exp\Parens{
      -\frac{\l^{1-\d}nr^2}{2\k_{\d}^2(1+r/3)}
    }
    \,.
    \qedhere
  \]
\end{proof}

\begin{lemma}
  \label{lem:moment}
  \[
    \E\Brackets{ \norm{\S-\T}^2 } \ \leq \
    \frac{34\k^4}{n} + \frac{15\k^4}{n^2}
    \,.
  \]
\end{lemma}
\begin{proof}
  The proof is an application of Lemma~\ref{lem:minsker-moment}.
  Define, for $1 \leq i \leq n$, $\X_i = \frac1n \parens{ \bphi_i\bphi_i^\top - \T }$, and also
  define $\Y = \sum_{i=1}^n \X_i$, so $\Y = \S - \T$.
  Clearly $\E\brackets{\X_i} = \mathbf0$ and $\norm{\X_i} \leq \k^2/n$.
  Additionally, since $\E\brackets{ \Y^2 } = n^{-1} \E\brackets{ \norm{\bphi_1}^2
  \bphi_1\bphi_1^\top - \T^2 }$, we have $\norm{\E\brackets{ \Y^2 }} \leq \k^4/n$ and
  $\tr\parens{\E\brackets{ \Y^2 }} \leq \k^4/n$.
  The claim thus follows by applying Lemma~\ref{lem:minsker-moment} with $V = \k^4/n$, $D=1$, and
  $R=\k^2/n$.
\end{proof}

\subsection{Bias bound}

\begin{lemma}
  \label{lem:bias}
  Assume that $\bb = \T^{\zeta/2}\ba$ for some $\zeta \geq 0$ and that $g_{\l}$ has qualification
  at least $\max\braces{(\zeta+1)/2,1}$.
  For any $0 < r \leq 1/2$,
  \begin{align*}
    \bias(\estl)
    & \ \leq \
    \frac{2^{\zeta+2}}{1-r}\norm{\ba}^2 \l^{\zeta+1}
    + t_1^2 \cdot \norm{\bb}^2 \cdot P(\cA_r)
    \\
    & \qquad + \ind{\zeta>1} \cdot \frac{8\zeta^2 (1+r)^{\zeta-1}}{1-r}
    \cdot \norm{\ba}^2
    \cdot \norm{\T+\l\I}^{\zeta-1}
    \cdot \E\Brackets{\norm{\S-\T}^2}
    \,.
  \end{align*}
\end{lemma}
\begin{proof}
  Define $h_{\l}(\S) = \I - \S \linv(\S)$.
  Since $\abs{1-s\linv(s)} \leq 1$ for $0 < s \leq \k^2$ and $\norm{\S} \leq \k^2$, we have
  \[
    \norm{h_{\l}(\S)}
    \ \leq \
    1
    \,.
  \]
  Moreover, using $\bb = \T^{\zeta/2}\ba$,
  \begin{align*}
    \bias(\estl)
    & \ = \
    \E\Brackets{ \Norm{\T^{1/2}h_{\l}(\S)\bb}^2 }
    \\
    & \ \leq \
    \E\Brackets{ \Norm{\T^{1/2}h_{\l}(\S)\T^{\zeta/2}}^2 \cdot \ind{\cA_d^c} } \cdot \norm{\ba}^2
    + \norm{\T^{1/2}}^2 \cdot \norm{h_{\l}(\S)}^2 \cdot \norm{\bb}^2 \cdot P(\cA_d)
    \\
    & \ \leq \
    \E\Brackets{ \Norm{\T^{1/2}h_{\l}(\S)\T^{\zeta/2}}^2 \cdot \ind{\cA_d^c} } \cdot \norm{\ba}^2
    + t_1^2 \cdot \norm{\bb}^2 \cdot P(\cA_d)
    \,.
  \end{align*}
  The rest of the proof involves bounding $\E\brackets{ \norm{\T^{1/2}h_{\l}(\S)\T^{\zeta/2}}^2
  \cdot \ind{\cA_r^c} }$.
  We separately consider two cases: (i) $\zeta \leq 1$, and (ii) $\zeta > 1$.

  \emph{Case 1: $\zeta \leq 1$.}
  Since $\linv$ has qualification at least $1$, it follows that $\abs{(s+\l)(1-s\linv(s))} \leq 2\l$
  for $0 < s \leq \k^2$.
  This implies
  \begin{align}
    \Norm{\T^{1/2}h_{\l}(\S)\T^{\zeta/2}}^2
    & \ \leq \
    \Norm{\T^{1/2}(\S+\l\I)^{-1/2}}^2
    \cdot \Norm{(\S+\l\I)h_{\l}(\S)}^2
    \cdot \Norm{\T^{\zeta/2}(\S+\l\I)^{-1/2}}^2
    \notag \\
    & \ \leq \
    4\l^2
    \cdot \Norm{\T^{1/2}(\S+\l\I)^{-1}\T^{1/2}}
    \cdot \Norm{\T^{\zeta/2}(\S+\l\I)^{-1}\T^{\zeta/2}}
    \,.
    \label{eq:bias-matrix1}
  \end{align}
  For $0 \leq z \leq 1$,
  \begin{align}
      \Norm{\T^{z/2}(\S+\l\I)^{-1}\T^{z/2}}
    & \ = \
    \Norm{\T^{z/2}\Braces{\Parens{\S-\T}+\Parens{\T+\l\I}}^{-1}\T^{z/2}}
    \notag \\
    & \ \leq \
    \Norm{\T^{z/2}(\T+\l\I)^{-1/2}}^2
    \cdot \Norm{(\T+\l\I)^{1/2}\Braces{\Parens{\S-\T}+\Parens{\T+\l\I}}^{-1}(\T+\l\I)^{1/2}}
    \notag \\
    & \ = \
    \Norm{\T^z(\T+\l\I)^{-1}}
    \cdot \Norm{
      \Parens{
        \I - (\T+\l\I)^{-1/2} \Parens{\T-\S} (\T+\l\I)^{-1/2}
      }^{-1}
    }
    \notag \\
    & \ \leq \
    \l^{z-1}
    \cdot \Norm{
      \Parens{
        \I - (\T+\l\I)^{-1/2} \Parens{\T-\S} (\T+\l\I)^{-1/2}
      }^{-1}
    }
    \,,
    \label{eq:bias-matrix2}
  \end{align}
  where the final inequality uses the fact $s^z/(s+\l) \leq \l^{z-1}$ for $0 \leq z \leq 1$ and $s
  \geq 0$.
  The final quantity in~\eqref{eq:bias-matrix2} is bounded above by $\l^{z-1} / (1-r)$ on the event
  $\cA_r^c$, so applying this inequality with $z=1$ and $z=\zeta$ to~\eqref{eq:bias-matrix1} gives
  \begin{align*}
    \Norm{\T^{1/2}h_{\l}(\S)\T^{\zeta/2}}^2 \cdot \ind{\cA_r^c}
    & \ \leq \
    4\l^2 \cdot \frac{1}{1-r} \cdot \frac{\l^{\zeta-1}}{1-r}
    \ = \
    \frac{4\l^{\zeta+1}}{(1-r)^2}
    \,.
  \end{align*}
  So the bias in this case is bounded as
  \[
    \bias(\estl)
    \ \leq \
    \Parens{ \frac{2}{1-r} }^2 \cdot \norm{\ba}^2 \cdot \l^{\zeta+1}
    + t_1^2 \cdot \norm{\bb}^2 \cdot P(\cA_r)
    \,.
  \]

  \emph{Case 2: $\zeta>1$.}
  We have
  \begin{align}
    \Norm{\T^{1/2}h_{\l}(\S)\T^{\zeta/2}}^2
    & \ \leq \
    \Norm{\T^{1/2}(\S+\l\I)^{-1/2}}^2
    \cdot \Norm{(\S+\l\I)^{1/2}h_{\l}(\S)\T^{\zeta/2}}^2
    \notag \\
    & \ = \
    \Norm{\T^{1/2}(\S+\l\I)^{-1}\T^{1/2}}
    \cdot \Norm{(\S+\l\I)^{1/2}h_{\l}(\S)\T^{\zeta/2}}^2
    \notag \\
    & \ \leq \
    \Norm{\T^{1/2}(\S+\l\I)^{-1}\T^{1/2}}
    \cdot \Norm{(\S+\l\I)^{1/2}h_{\l}(\S)(\T+\l\I)^{\zeta/2}}^2
    \,.
    \label{eq:bias-matrix3}
  \end{align}
  The first factor on the right-hand side of~\eqref{eq:bias-matrix3} can be bounded
  using~\eqref{eq:bias-matrix2} on the event $\cA_r^c$.
  For the second factor, we have
  \begin{align}
    \lefteqn{
      \Norm{
        (\S+\l\I)^{1/2} h_{\l}(\S) (\T+\l\I)^{\zeta/2}
      }
    }
    \notag \\
    & \ \leq \
    \Norm{
      (\S+\l\I)^{(\zeta+1)/2} h_{\l}(\S)
    }
    + \Norm{
      (\S+\l\I)^{1/2} h_{\l}(\S)
      \Braces{
        (\T+\l\I)^{\zeta/2}
        - (\S+\l\I)^{\zeta/2}
      }
    }
    \notag \\
    & \ \leq \
    (2\l)^{(\zeta+1)/2}
    + 2\l^{1/2} \cdot
    \Norm{ (\T+\l\I)^{\zeta/2} - (\S+\l\I)^{\zeta/2} }
    \,.
    \label{eq:bias-matrix4}
  \end{align}
  Above, the first inequality is due to the triangle inequality, and the second inequality uses the
  facts that $\linv$ has qualification at least $(\zeta+1)/2$, and that $(s+\l)^{(z+1)/2} \leq
  2^{-1+(z+1)/2}(s^{(z+1)/2}+\l^{(z+1)/2})$ for $s\geq0$ and $z\geq1$.

  We now bound $\norm{ (\T+\l\I)^{\zeta/2} - (\S+\l\I)^{\zeta/2} }$ in terms of $\norm{\T-\S}$.
  First, observe that on the event $\cA_r^c$,
  \begin{align*}
    \norm{\T-\S}
    & \ \leq \
    \norm{\T+\l\I} \cdot
    \norm{(\T+\l\I)^{-1/2}(\T-\S)(\T+\l\I)^{-1/2}}
    \ < \
    r \cdot \norm{\T+\l\I}
    \,,
  \end{align*}
  and, consequently,
  \begin{align*}
    \norm{\S+\l\I}
    & \ \leq \
    (1+r) \cdot \norm{\T+\l\I}
    \,.
  \end{align*}
  For a small constant $s>0$, define
  \begin{align*}
    \A_s & \ = \ \frac1{(1+r+s) \norm{\T+\l\I}} \cdot (\T+\l\I) \,, \\
    \B_s & \ = \ \frac1{(1+r+s) \norm{\T+\l\I}} \cdot (\S+\l\I) \,.
  \end{align*}
  Then, on the event $\cA_r^c$, applying Lemma~\ref{lem:power} and Lemma~\ref{lem:fractional} gives
  \begin{align*}
    \Norm{\A_s^{\zeta/2} - \B_s^{\zeta/2}}
    & \ \leq \
    2\zeta \cdot \Norm{\A_s^{1/2} - \B_s^{1/2}}
    \\
    & \ \leq \
    \zeta \cdot \Braces{ \frac{\l}{(1+r+s)\norm{\T+\l\I}} }^{-1/2}
    \cdot \Norm{\A_s-\B_s}
    \\
    & \ = \
    \zeta \cdot \Braces{ \l \cdot (1+r+s)\norm{\T+\l\I} }^{-1/2}
    \cdot \Norm{\T-\S}
    \,.
  \end{align*}
  Taking $s\to0$, it follows that on $\cA_r^c$,
  \begin{align}
    \Norm{ (\T+\l\I)^{\zeta/2} - (\S+\l\I)^{\zeta/2} }
    & \ \leq \
    \zeta \cdot \l^{-1/2} \cdot \Braces{ (1+r)\norm{\T+\l\I} }^{(\zeta-1)/2}
    \cdot \Norm{\T-\S}
    \,.
    \label{eq:bias-matrix5}
  \end{align}

  Combining~\eqref{eq:bias-matrix4} and~\eqref{eq:bias-matrix5} gives
  \begin{align*}
    \Norm{ (\S+\l\I)^{1/2} h_{\l}(\S) (\T+\l\I)^{\zeta/2} }
    & \ \leq \
    (2\l)^{(\zeta+1)/2}
    + 2\zeta \cdot \Braces{ (1+r)\norm{\T+\l\I} }^{(\zeta-1)/2}
    \cdot \Norm{\T-\S}
    \,.
  \end{align*}
  Using this together with~\eqref{eq:bias-matrix2} in~\eqref{eq:bias-matrix3}
  and
  \begin{align*}
      \Norm{\T^{1/2}h_{\l}(\S)\T^{\zeta/2}}^2 \cdot \ind{\cA_r^c}
    & \ \leq \
    \frac1{1-r}
    \cdot \Parens{
      (2\l)^{(\zeta+1)/2}
      + 2\zeta \cdot \Braces{ (1+r)\norm{\T+\l\I} }^{(\zeta-1)/2}
      \cdot \Norm{\T-\S}
    }^2
    \,.
  \end{align*}
  Therefore
  \begin{equation*}
    \bias(\estl)
    \ \leq \
    \frac{\norm{\ba}^2}{1-r}
    \cdot \Parens{
      2 \cdot (2\l)^{\zeta+1}
      + 8\zeta^2 \cdot \Braces{ (1+r)\norm{\T+\l\I} }^{\zeta-1}
      \cdot \E\Brackets{ \Norm{\T-\S}^2 }
    }
    + t_1^2 \cdot \norm{\bb}^2 \cdot P(\cA_r)
    \,.
    \qedhere
  \end{equation*}
\end{proof}

\subsection{Variance bound}

\begin{lemma}
  \label{lem:variance}
  For any $0 < r < 1$,
  \begin{align*}\label{newVbd}
    \variance(\estl)
    & \ \leq \
    \frac{2}{1-r} \cdot \frac{\dl\s^2}{n}
    + \frac{\k^2\s^2}{\l n} \cdot P(\cA_r)
    \,.
  \end{align*}
\end{lemma}
\begin{proof}
  The assumption on $\ee = (\e_1,\dotsc,\e_n)^\top$ implies $\E\brackets{\e_i} = 0$,
  $\E\brackets{\e_i\e_j} = 0$ for $i \neq j$, and $\E\brackets{\e_i^2} \leq \sigma^2$.
  So, by Von Neumann's inequality,
  \begin{align*}
    \variance(\estl)
    & \ = \ \frac{1}{n^2}\E\Brackets{ \Norm{ \T^{1/2}\linv(\S)\bPhi^{\top}\ee }^2 }
    \ \leq \
    \frac{\sigma^2}{n}\E\Brackets{ \tr\Parens{ \T\linv(\S)^2\S } }
    \,.
  \end{align*}
  Using Von Neumann's inequality together with $\tr(\T) \leq \k^2$ and $\linv(s)^2s \leq 1/\l$ for
  $0 < s \leq \k^2$,
  \begin{align*}
    \tr\Parens{ \T\linv(\S)^2\S }
    & \ \leq \
    \tr(\T) \Norm{\linv(\S)^2\S}
    \ \leq \
    \frac{\k^2}{\l}
    \,.
  \end{align*}
  Therefore
  \begin{align*}
    \variance(\estl)
    & \ \leq \
    \frac{\sigma^2}{n}\E\Brackets{ \tr\Parens{ \T\linv(\S)^2\S } \cdot \ind{\cA_r^c} }
    + \frac{\k^2\s^2}{\l n} \cdot P(\cA_r)
    \,.
  \end{align*}
  Using Von Neumann's inequality twice more, and $(s+\l)\linv(s)^2s \leq 2$ for $0 < s \leq
  \k^2$, we have
  \begin{align*}
    \tr\Parens{ \T\linv(\S)^2\S } 
    & \ = \
    \tr\Parens{ \T(\S+\l\I)^{-1}(\S+\l\I)\linv(\S)^2\S } 
    \\
    & \ \leq \
    \tr\Parens{ \T(\S+\l\I)^{-1} } \Norm{ (\S+\l\I)\linv(\S)^2\S } 
    \\
    & \ \leq \
    2\tr\Parens{ \T(\S+\l\I)^{-1} }
    \\
    & \ = \
    2\tr\Parens{ \T(\T+\l\I)^{-1}(\T+\l\I)^{1/2}(\S+\l\I)^{-1}(\T+\l\I)^{1/2} }
    \\
    & \ \leq \
    2\tr\Parens{ \T(\T+\l\I)^{-1} } \Norm{ (\T+\l\I)^{1/2}(\S+\l\I)^{-1}(\T+\l\I)^{1/2} }
    \\
    & \ = \
    2\dl \Norm{ \Parens{ \I - (\T+\l\I)^{1/2}(\T-\S)^{-1}(\T+\l\I)^{1/2}
    }^{-1} }
    \,.
  \end{align*}
  This final quantity is bounded above by $2\dl/(1-r)$ on the event $\cA_d^c$.
\end{proof}

\subsection{Finishing the proof of Theorem~\ref{thm:UB}}

Using the bias-variance decomposition from Proposition~\ref{prop:bv}, we apply the bias bound
(Lemma~\ref{lem:bias}) and variance bound (Lemma~\ref{lem:variance}) to obtain a bound on the risk:
\begin{align*}
  \risk(\estl)
  & \ = \
  \bias(\estl) + \variance(\estl)
  \\
  & \ \leq \
  \frac{2^{\zeta+2}}{1-r}\norm{\ba}^2 \l^{\zeta+1}
  + t_1^2 \cdot \norm{\bb}^2 \cdot P(\cA_r)
  \\
  & \qquad + \ind{\zeta>1} \cdot \frac{8\zeta^2 (1+r)^{\zeta-1}}{1-r}
  \cdot \norm{\ba}^2
  \cdot \parens{ t_1^2 + \l }^{\zeta-1}
  \cdot \E\Brackets{\norm{\S-\T}^2}
  \\
  & \qquad + \frac{2}{1-r} \cdot \frac{\dl\s^2}{n} + \frac{\k^2\s^2}{\l n} \cdot P(\cA_r)
  \,.
\end{align*}
Now set $r = 1/2$, and apply Lemma~\ref{lem:tail} to bound $P(\cA_r)$.  Note that the assumption
$(8/3+2\sqrt{5/3}) \k_{\d}^2 / n \leq \l^{1-\d}$ satisfies the conditions on $r$ in
Lemma~\ref{lem:tail} for $r=1/2$.
Finally, apply Lemma~\ref{lem:moment} to bound
$\E\brackets{\norm{\S-\T}^2}$.

\section{Proof of Proposition~\ref{prop:finite_adapt}}
\label{sec:proof-finite_adapt}
\newcommand\wh\widehat
\newcommand\hU{\widehat{\U}}

Let $\hat{t}_1^2 \geq \hat{t}_2^2 \geq \cdots \geq 0$  denote the eigenvalues of
$\S$ and define $\wh{\T} = \diag(\hat{t}_1^2,\hat{t}_2^2,\dotsc)$.
Let $\hU$ be an orthogonal transformation satisfying $\S =
\hU\wh{\T}\hU^{\top}$.
Additionally, let $h_{\l}(t) = I\{t \leq \l\}$, define $\hat{J} = \hat{J}_{\l} =
\inf\braces{j : \hat{t}_j^2 > \l}$, and write $\hU = (\hU_- \ \ \hU_+)$,
where $\hU_-$ is the $\infty \times \hat{J}$ matrix comprised of the first
$\hat{J}$ columns of $\hU$, and $\hU_+$ consists of the remaining
columns of $\hU$.
Finally, define $\T_- = \diag(t_1^2,\ldots,t_J^2)$, $\wh{\T}_+ =
\diag(\hat{t}_{\hat{J}+1}^2,\hat{t}_{\hat{J}+2}^2,\dotsc)$, and define the
$\infty \times J$ matrix $\U_- = (\I_J \ \ \mathbf{0})^{\top}$, where $\I_J$ is the $J
\times J$ identity matrix.  

Now consider the bias term $\bias(\kpcr) = \E\brackets{
  \norm{\T^{1/2}h_{\l}(\S)\bb}^2 }$, and observe that
\[
  \Norm{
    \T^{1/2}h_{\l}(\S)\bb
  }^2
  \ = \
  \Norm{
    \T^{1/2}\hU_{+}\hU_{+}^{\top}\U_-\U_-^{\top}\bb
  }^2
  \ \leq  \
  \k^2\Norm{
    \hU_{+}^{\top}\U_-
  }^2
  \norm{f^{\dagger}}_{\H}^2
  \,.
\]
Thus, 
\begin{equation}\label{pbound}
  \bias(\kpcr)
  \ \leq \
  \k^2 \norm{f^{\dagger}}_{\H}^2
  \cdot
  \E\Parens{
    \Norm{
      \hU_{+}^{\top}\U_-
    }^2 \cdot \ind{\hat{J} \geq J}
  } + \k^2\norm{f^{\dagger}}_{\H}^2
  \cdot P(\hat{J} < J)
  \,.
\end{equation}
Now we bound $\Vert\hU_{+}^{\top}\U_-\Vert$ on the event $\{\hat{J}\geq J\}$.
We derive this bound from basic principles, but it is essentially the
Davis-Kahan inequality~\citep{davis1970rotation}.
Let $\bD =\S - \T$.  
Then 
\[
  \bD \U_-
  \ = \
  \S \U_- - \T \U_-
  \ = \
  \S \U_- - \U_- \T_-
\]
and 
\[
  \U_-^{\top} \bD \hU_+
  \ = \
  \U_-^{\top}\S \hU_+ - \T_- \U_-^{\top} \hU_{+}
  \ = \
  \U_-^{\top} \hU_+ \wh{\T}_+ - \T_-\U_-^{\top} \hU_+
  \,.
\]
Next notice that
\begin{align*}
  \norm{\bD}
  & \ \geq \
  \norm{ \U_-^{\top} \bD \hU_+ }
  \ \geq \
  \norm{ \T_- \U_-^{\top} \hU_+ } - \norm{ \U_-^{\top} \hU_+ \wh{\T}_+ }
  \\
  & \ \geq \
  t_J^2
  \norm{ \U_-^{\top} \hU_+ }
  - (1  - r)t_J^2
  \norm{ \U_-^{\top} \hU_+ }
  \ = \
  rt_J^2 \norm{ \U_-^{\top} \hU_+ }
  \,.
\end{align*}
Thus,
\begin{equation}\label{DKb}
  \norm{ \U_-^{\top} \hU_+ }
  \ \leq \ \frac{1}{rt_J^2} \norm{ \bD }
\end{equation}
on the event $\braces{ \hat{J} > J }$.

Next we bound $P(\hat{J} < J) = P(\hat{t}^2_J \leq \l) = P\braces{\hat{t}^2_J
\leq (1 - r) t_J^2}$.
By Weyl's inequality, 
\[
  \abs{ \hat{t}_J^2 - t_J^2 }
  \ \leq \
  \norm{\bD}
  \ = \
  \norm{ \S - \T }
  \,,
\]
and, by Lemma~\ref{lem:minsker}, 
\[
  P\Parens{
    \norm{\S  - \T}
    \geq rt_J^2
  }
  \ \leq \
  4\exp\Parens{
    -\frac{n r^2t_J^4}{2(\k^4 + \k^2rt_J^2/3)}
  }
  \,, 
\]
provided $rt_J^2 \geq \k^2/n^{1/2} + \k^2/(3n)$.  
Thus, 
\begin{equation}\label{LLD}
  P(\hat{J} < J)
  \ = \
  P\Braces{
    \hat{t}^2_J \leq (1 - r) t_J^2
  }
  \ \leq \
  4\exp\Parens{
    -\frac{n r^2t_J^4}{2(\k^4 + \k^2rt_J^2/3)}
  }
  \,.
\end{equation}

Combining \eqref{pbound}--\eqref{LLD} and using Lemma~\ref{lem:moment} gives 
\begin{align*}
  \bias(\kpcr)
  & \ \leq \
  \frac{\k^2}{r^2t_J^4}
  \norm{f^{\dagger}}_{\H}^2
  \E\Brackets{ \norm{ \S - \T }^2 }
  + 4\k^2 \norm{f^{\dagger}}_{\H}^2
  \exp\Parens{
    -\frac{nr^2t_J^4}{2(\k^4 + \k^2rt_J^2/3)}
  }
  \\
  & \ \leq \
  \frac{\k^2}{r^2t_J^4}
  \norm{f^{\dagger}}_{\H}^2
  \Parens{\frac{34\k^4}{n} + \frac{15\k^4}{n^2}}
  + 4\k^2\norm{f^{\dagger}}_{\H}^2
  \exp\Parens{-\frac{nr^2t_J^4}{2(\k^4 + \k^2rt_J^2/3)} }
  \,.
\end{align*}
Next, we combine this bound on $\mathrm{B}_{\rho}(\kpcr)$ with the
variance bound
Lemma \ref{lem:variance} (taking $r = 0$ in the lemma) to
obtain
\begin{align*}
  \risk(\kpcr)
  & \ = \
  \bias(\estl)+ \variance(\estl)
  \\
  & \ \leq \
  \frac{\k^2}{r^2t_J^4} \norm{f^{\dagger}}_{\H}^2 \Parens{\frac{34\k^4}{n} + \frac{15\k^4}{n^2}}
  + 4\k^2\norm{f^{\dagger}}_{\H}^2 \exp\Parens{-\frac{nr^2t_J^4}{2(\k^4 + \k^2rt_J^2/3)} }
  + \frac{2\dl\s^2}{n} + \frac{\k^2\s^2}{\l n}
  \\
  & \ \leq \
  \frac{\k^2}{r^2t_J^4} \norm{f^{\dagger}}_{\H}^2 \Parens{\frac{34\k^4}{n} + \frac{15\k^4}{n^2}}
  + 4\k^2\norm{f^{\dagger}}_{\H}^2 \exp\Parens{-\frac{nr^2t_J^4}{2(\k^4 + \k^2rt_J^2/3)} }
  + \frac{3\kappa^2\s^2}{(1-r)t_J^2n}
  \\
  & \ = \
  \frac1n\Braces{
    \frac{34\k^6}{r^2t_J^4} \norm{f^{\dagger}}_{\H}^2 
    + \frac{3\kappa^2}{(1-r)t_J^2} \s^2
  }
  +
  \k^2 \norm{f^{\dagger}}_{\H}^2
  \Braces{
    \frac{15\k^4}{r^2t_J^4n^2}
    + 4\exp\Parens{-\frac{nr^2t_J^4}{2(\k^4 + \k^2rt_J^2/3)} }
  }
  \,.
\end{align*}
This completes the proof of the proposition.

\section{Supporting results}

\subsection{Sums of random operators}

\begin{lemma}
  \label{lem:minsker}
  Let $R>0$ be a positive real constant and consider a finite sequence of self-adjoint Hilbert-Schmidt
  operators $\braces{ \X_i }_{i=1}^n$ satisfying $\E\brackets{\X_i} = \mathbf0$ and $\norm{\X_i}
  \leq R$ almost surely.
  Define $\Y = \sum_{i=1}^n \X_i$ and suppose there are constants $V,D > 0$ satisfying
  $\norm{\E\brackets{ \Y^2 }} \leq V$ and $\tr\parens{\E\brackets{\Y^2}} \leq VD$.
  For all $t \geq V^{1/2} + R/3$,
  \[
    P\Parens{ \norm{\Y} \geq t }
    \ \leq \
    4D\exp\Parens{ -\frac{t^2}{2(V+Rt/3)} }
    \,.
  \]
\end{lemma}
\begin{proof}
  This is a straightforward generalization of~\citep[Theorem 7.7.1]{tropp2015intro}, using the
  arguments from \citep[Section 4]{minsker2011bernstein} to extend from self-adjoint matrices to
  self-adjoint Hilbert-Schmidt operators.
\end{proof}

\begin{lemma}
  \label{lem:minsker-moment}
  In the same setting as Lemma~\ref{lem:minsker},
  \[
    \E\Brackets{ \norm{\Y}^2 }
    \ \leq \
    (2+32D)V + \Parens{ \frac{2+128D}{9} } R^2
    \,.
  \]
\end{lemma}
\begin{proof}
  The proof is based on integrating the tail bound from Lemma~\ref{lem:minsker}:
  \begin{align*}
    \lefteqn{
      \E\Brackets{ \norm{\Y}^2 }
      \ = \ \int_0^\infty P\Parens{ \norm{\Y} \geq t^{1/2} } \dif t
    } \\
    & \ \leq \
    \Parens{ V^{1/2} + \frac{R}{3} }^2
    + \int_{\Parens{ V^{1/2} + \frac{R}{3} }^2}^\infty 4D\exp\Parens{ -\frac{t}{2(V+Rt^{1/2}/3)} } \dif t
    \\
    & \ \leq \
    \Parens{ V^{1/2} + \frac{R}{3} }^2
    + \int_0^{9V^2/R^2} 4D\exp\Parens{ -\frac{t}{4V} } \dif t
    + \int_{9V^2/R^2}^\infty 4D\exp\Parens{ -\frac{3t^{1/2}}{4R} } \dif t
    \\
    & \ = \
    \Parens{ V^{1/2} + \frac{R}{3} }^2
    + 16VD\Braces{ 1 - \exp\Parens{-\frac{9V}{4R^2}} }
    + \frac{128R^2D}{9} \Braces{ \frac{9V}{4R^2} + 1 } \exp\Parens{-\frac{9V}{4R^2}}
    \\
    & \ = \
    \Parens{ V^{1/2} + \frac{R}{3} }^2
    + 16VD\Braces{ 1 + \exp\Parens{-\frac{9V}{4R^2}} }
    + \frac{128R^2D}{9} \exp\Parens{-\frac{9V}{4R^2}}
    \\
    & \ \leq \
    (2+32D)V + \Parens{ \frac{2+128D}{9} } R^2
    \,.
    \qedhere
  \end{align*}
\end{proof}

\subsection{Differences of powers of bounded operators}

\begin{lemma}
  \label{lem:power}
  Let $\A$ and $\B$ be non-negative self-adjoint operators with $\norm{\A} < 1$ and $\norm{\B} < 1$.
  For any $\g \geq 1$,
  \[
    \norm{\A^\g - \B^\g}
    \ \leq \
    2\g\norm{\A - \B}
    \,.
  \]
\end{lemma}
\begin{proof}
  The proof considers three possible cases for the value of $\g$:
  (i) $\g$ is an integer, (ii) $1 < \g < 2$, and (iii) $\g$ is a non-integer
  larger than two.

  \emph{Case 1: $\g$ is an integer.}
  In this case, we have the following identity:
  \[
    \A^\g - \B^\g
    \ = \
    \sum_{j=1}^\g \A^{j-1} \parens{ \A - \B } \B^{\g-j}
    \,.
  \]
  So, by the triangle inequality,
  \[
    \norm{\A^\g - \B^\g}
    \ \leq \
    \norm{\A-\B}
    \sum_{j=1}^\g \norm{\A^{j-1}} \norm{\B^{\g-j}}
    \ \leq \
    \g\norm{\A-\B}
    \,.
  \]

  \emph{Case 2: $1 < \g < 2$.}
  Pick $0 < r < 1$ such that $\norm{\A} < r$ and $\norm{\B} < r$, and fix $0 < t < (1-r)/2$.
  Define $\A_t = \A + t\I$ and $\B_t = \B + t\I$, and
  \[
    \binom{\g}{k} \ = \
    \frac{\g(\g-1)\dotsb(\g-k+1)}{k!}
    \,,
    \quad k = 1, 2, \dotsc
    \,.
  \]
  Then, using the power series $(1+s)^\g = 1 + \sum_{k=1}^\infty \binom{\g}{k} s^k$ for $-1 < s <
  1$,
  \[
    \A_t^\g - \B_t^\g
    \ = \
    \sum_{k=1}^\infty
    \binom{\g}{k}
    \braces{ (\A_t - \I)^k - (\B_t - \I)^k }
    \ = \
    \sum_{k=1}^\infty
    \binom{\g}{k}
    \sum_{j=1}^k
    \A^{j-1} \parens{ \A - \B } \B^{k-j}
    \,.
  \]
  Convergence is assured because $\norm{\A_t-\I} \leq 1-t$ and $\norm{\B_t-\I} \leq 1-t$.
  Moreover,
  \begin{align*}
    \norm{\A_t^\g - \B_t^\g}
    & \ \leq \
    \norm{\A-\B}
    \sum_{k=1}^\infty
    \Abs{\binom{\g}{k}}
    \sum_{j=1}^k \norm{\A_t-\I}^{j-1} \norm{\B_t-\I}^{k-j}
    \\
    & \ \leq \
    \norm{\A-\B}
    \sum_{k=1}^\infty
    \Abs{\binom{\g}{k}}
    k(1-t)^{k-1}
   \ = \
    \g
    \norm{\A-\B}
    \Braces{
      1 + \sum_{k=1}^\infty \Abs{\binom{\g-1}{k}} (1-t)^k
    }
    \\
    & \ = \
    \g
    \norm{\A-\B}
    \Braces{ 2 - (1-(1-t))^{\g-1} }
  \ \leq \
    2\g\norm{\A-\B}
    \,.
  \end{align*}
  Above, we have used the power series $1-(1-s)^{\g-1} = \sum_{k=1}^\infty \abs{\binom{\g-1}{k}}
  s^k$ at $s = 1-t$.
  Taking $t\to0$ yields
  \[
    \norm{\A^\g - \B^\g} \ \leq \ 2\g\norm{\A-\B}
    \,.
  \]

  \emph{Case 3: $\g$ is a non-integer larger than two.}
  We can write $\g = k + q$ for an integer $k\geq2$ and real number $0 < q < 1$.
  Applying the results from the previous two cases gives
  \[
    \norm{\A^\g - \B^\g}
    \ = \
    \norm{(\A^k)^{\g/k} - (\B^k)^{\g/k}}
    \ \leq \
    \frac{2\g}{k} \norm{\A^k - \B^k}
    \ \leq \
    2\g \norm{\A - \B}
    \,.
    \qedhere
  \]
\end{proof}

\begin{lemma}
  \label{lem:fractional}
  Pick any real numbers $r,\g \in (0,1)$.
  Let $\A$ and $\B$ be non-negative self-adjoint operators, each with spectrum contained in $\intco{r,1}$.
  Then
  \[
    \norm{\A^\g - \B^\g}
    \ \leq \
    \g r^{\g-1}\norm{\A - \B}
    \,.
  \]
\end{lemma}
\begin{proof}
  Since $\norm{\A-\I} \leq 1-r$ and $\norm{\B-\I} \leq 1-r$, the proof is similar to that of
  Lemma~\ref{lem:power}:
  \begin{align*}
    \norm{\A^\g - \B^\g}
    & \ \leq \
    \g\norm{\A-\B}
    \Braces{
      1 +
      \sum_{k=1}^\infty \Abs{\binom{\g-1}{k}} (1-r)^k
    }
    \ = \
    \g r^{\g-1}\norm{\A-\B}
    \,.
  \end{align*}
  Above, we have used the power series $(1-s)^{\g-1} = 1 + \sum_{k=1}^\infty \abs{\binom{\g-1}{k}}
  s^k$ at $s = 1-r$ (which differs from the case in Lemma~\ref{lem:power} because $0 < \g < 1$).
\end{proof}

\end{document}